\documentclass[11pt]{article}
\usepackage{amssymb,amsfonts,amsmath,curves}
\usepackage{mathrsfs}
\usepackage{epsfig}
\usepackage{graphicx}
\usepackage{color}
\usepackage{anysize}
\usepackage[affil-it]{authblk}
 \usepackage[bb=boondox]{mathalfa}
 \usepackage{soul}
 \usepackage{enumerate}
 \usepackage[dvipsnames]{xcolor}
 \usepackage{hyperref}
\newcommand{\dd}{\mathrm{d}}
\providecommand{\mb}[1]{\mathbb{#1}}

\providecommand{\abs}[1]{\lvert#1\rvert}

\providecommand{\norm}[1]{\lVert#1\rVert}

\def\pf{{\it Proof:}~}

\marginsize{2cm}{2cm}{1cm}{2cm}

\newtheorem{defin}{Definition}[section]

\newtheorem{theorem}{Theorem}[section]

\newtheorem{lemma}{Lemma}[section]

\newtheorem{con}{Conjecture}
\newtheorem{cond}{Condition}
\newtheorem{conj}{Conjecture}
\newtheorem{prop}{Proposition}[section]
\newtheorem{lem}{Lemma}[section]
\newtheorem{rmk}{Remark}[section]
\newtheorem{corollary}{Corollary}[section]

\newcommand{\beq}{\begin{equation}}
\newcommand{\eeq}{\end{equation}}
\newcommand{\beqn}{\begin{eqnarray}}
\newcommand{\beqnn}{\begin{eqnarray*}}
\newcommand{\eeqn}{\end{eqnarray}}
\newcommand{\eeqnn}{\end{eqnarray*}}
\newcommand{\bprop}{\begin{prop}}
\newcommand{\eprop}{\end{prop}}
\newcommand{\bteo}{\begin{teo}}
\newcommand{\bcor}{\begin{cor}}
\newcommand{\ecor}{\end{cor}}
\newcommand{\bcon}{\begin{con}}
\newcommand{\econ}{\end{con}}
\newcommand{\bcond}{\begin{cond}}
\newcommand{\econd}{\end{cond}}
\newcommand{\bconj}{\begin{conj}}
\newcommand{\econj}{\end{conj}}
\newcommand{\eteo}{\end{teo}}
\newcommand{\brm}{\begin{rmk}}
\newcommand{\erm}{\end{rmk}}
\newcommand{\blem}{\begin{lem}}
\newcommand{\elem}{\end{lem}}
\newcommand{\ben}{\begin{enumerate}}
\newcommand{\een}{\end{enumerate}}
\newcommand{\bei}{\begin{itemize}}
\newcommand{\eei}{\end{itemize}}
\newcommand{\bdf}{\begin{defin}}
\newcommand{\edf}{\end{defin}}
\newcommand{\bpr}{\begin{proof}}
\newcommand{\epr}{\end{proof}}
\newenvironment{proof}{\noindent {\em Proof}.\,\,}{\hspace*{\fill}$\halmos$\medskip}

\newcommand{\halmos}{\rule{1ex}{1.4ex}}
\def \qed {{\hspace*{\fill}$\halmos$\medskip}}

\newcommand{\Z}{{\mathbb Z}}
\newcommand{\R}{{\mathbb R}}

\renewcommand{\P}{{\mathbb P}}
\newcommand{\Q}{\mathbb Q}
\newcommand{\N}{{\mathbb N}}

\renewcommand{\o}{\omega}

\renewcommand{\le}{\leq}

\definecolor{Red}{rgb}{1,0,0}

\title{Quenched invariance principle for random walk in
time-dependent  balanced random environment}
\author{Jean-Dominique Deuschel%
  \thanks{Electronic address: \texttt{deuschel@math.tu-berlin.de}}}
\affil{Instit\"ut f\"ur Mathematik\\ Technische Universit\"at Berlin}

\author{Xiaoqin Guo%
  \thanks{Electronic address: \texttt{guo297@purdue.edu}, Partially supported by an AMS-Simons Travel Grant.}}
\affil{Department of Mathematics\\ Purdue University}

\author{Alejandro F. Ram\'\i rez%
  \thanks{Electronic address: \texttt{aramirez@mat.uc.cl},
Partially supported by Iniciativa Cient\'\i fica Milenio NC120062,
Nucleus Millenium Stochastic Models of Complex and Disordered Systems
and by Fondo Nacional de Desarrollo Cient\'\i fico
y Tecnol\'ogico grant 1141094.
}}
\affil{Facultad de Matem\'aticas\\ Pontificia Universidad Cat\'olica de Chile}

\begin{document}
%\thanks{aaaa}

%$^1$ 
%}
\maketitle

\vskip-1cm

\let\thefootnote\relax

\let\thefootnote\relax

\abstract{We prove a quenched central limit theorem for balanced
random walks in time-dependent ergodic random environments which is not necessarily nearest-neighbor.
We assume that the environment satisfies  appropriate
ergodicity and 
ellipticity conditions. The proof is based on the use of
%%%GGG<
%the parabolic maximum principle of Krylov \cite{Kr-76}.}
a maximum principle for parabolic difference operators.
%%%GGG>

%\tableofcontents

%%%%%%%%%%%%%%%%%%%%%%%%%%%%%%%%%%%%%%%%%%%%%%%%%%%%%%%%%%%%%%%%%%%%%%%%%%%%%%%%%%%%%%%%%%%%%%%%%%%%%%%%%%%%%%%%%%%%%%%%%%%%%%%%%%%%%%%%%%

\section{Introduction}
We consider random walks in a balanced time-dependent 
random environment. Under a mild ergodicity
assumption on the law of the environment and a moment
condition on the jump probabilities,
we prove a quenched central limit theorem (QCLT). 
Our results extend previous results of Lawler \cite{L-82} and of
Guo and Zeitouni \cite{GZ-10} and are based on the use of
%%%GGG<
%the parabolic maximum principle of Krylov \cite{Kr-76}.}
a new maximum principle for parabolic difference operators.
%%%GGG>
Furthermore, they can be  considered as a version of discrete 
homogenization of stochastic parabolic operators in non-divergence form
without uniform ellipticity
(for homogenization results in a similar
 PDE settings, we refer to \cite{AS-14, Lin-15}).

We state our results in both discrete and continuous time settings.

\subsection{Discrete time RWRE}
Let $U$ be a nonempty finite subset  of $\mathbb Z^d$ which will be called the {\it jump range}.
Define a set of probability vectors
 \[
 \mathcal P=\mathcal P(U)
 :=\left\{ v=\{v( e)>0:e\in
U\}:\sum_{e\in U}v(e)=1\right\}.
 \]
Consider a discrete time stochastic process $\omega:=\{\omega_n:n\in
\mathbb N\}$
 with state space $\Omega:=\mathcal P^{\mathbb Z^d}$,
so that $\omega_n:=\{\omega_n(x):x\in\mathbb Z^d\}$ with
$\omega_n(x):=\{\omega_n(x,e):e\in U\}\in\mathcal P$.
 We call
$\Omega^{\mathbb N}$ the {\it environmental space} while an element
\begin{equation}
\label{environment}
\omega\in\Omega^{\mathbb N}
\end{equation} a {\it discrete time environment}.
Note that throughout this construction the set $U$ 
is fixed and is not dependent on $\omega$.
 Let us denote by $\mathbb
P$ the law of $\omega$ and $\mathbb E_{\mathbb P}$ its expectation.
 Given $\omega\in\Omega^{\mathbb N}$, $x\in\mathbb Z^d$ and
$n\in\mathbb Z$
 consider  the random walk $\{X_m:m\ge 0\}$
 with
a law $P_{x,n,\omega}$ on $(\mathbb Z^d)^{\mathbb N}$ defined through
$P_{x,n,\omega}(X_n=x)=1$ and the transition probabilities
\[
P_{x,n,\omega}(X_{n+k+1}=y+e|X_{n+k}=y)=\omega_{n+k}(y,e),
\]
for $k\ge 0$, $y\in\mathbb Z^d$ and $e\in U$.
We call this process a {\it discrete time  random walk in
time-dependent   random environment} and
 call 
$P_{x,n,\omega}$  the {\it quenched law} of
the  random walk starting from $x$ at time $n$.
The expectation of the law $P_{x,n,\omega}$ is denoted by $E_{x,n,\omega}$.

Given any topological space $T$, we will denote
by $\mathcal B(T)$ the corresponding Borel sets.
For  $x\in\mathbb Z^d$ and $n\ge 0$   define the
 space-time shift
 $\theta_{n,x}:\Omega^{\mathbb N}\to \Omega^{\mathbb N}$
by $(\theta_{n,x}\omega)_m(y):=\omega_{n+m}(y+x)$ for
all $m\ge 0$ and $y\in\mathbb Z^d$.
Throughout, we will assume that $\mathbb P$ is stationary
under the action of  $\{\theta_{n,x}:n\ge 0,x\in\mathbb Z^d\}$.
Let now $Z\subset \N\times\Z^d$. We will
say that
 $\{\theta_{n,e}:(n,e)\in Z\}$ is an ergodic
family of transformations for the probability space $(\Omega^{\mathbb N},
\mathcal B( \Omega^{\mathbb N}), \mathbb P)$
%or an {\it ergodic family of transformations with respect to $\mathbb P$},
if whenever $A\in\mathcal B( \Omega^{\mathbb N})$
satisfies $\theta^{-1}_{n,e} A=A$ for all $(n,e)\in\ Z,$ then $\mathbb P(A)\in\{0, 1\}$. 
Note that $\{\theta_{n,e}:(n,e)\in Z\}$
is an ergodic family of transformations for $\P$
if and only if   $\{\theta_{n,e}:(n,e)\in \langle Z\rangle\}$
is an ergodic family for $\P$,
where $\langle Z\rangle$ denotes the subset of $\mathbb N\times\mathbb Z^d$
generated 
by $Z$. 
%\sout{Note that we do not demand the environment to be necessarily ergodic
%under the time shifts.}

 When the environment is time-independent, i.e, $\omega_n=\omega_{n+1}$ for all $n\ge 0$, we 
 call $\omega$ a {\it static} environment.
% above random walk a {\it discrete time random walk
%in static random environment} and its corresponding
%law starting from $\omega$ at site $x$, just
%$P_{x,\omega}$.
In this case, we may drop the time subscripts. E.g, we may write $\omega_n(x,e), \theta_{n,x}$ and $P_{x,n,\omega}$ simply as $\omega(x,e), \theta_x$ and $P_{x,\omega}$.
%In the case of a static random environment, note
%that for $x\in\mathbb Z^d$,
% we can drop the time subscript from $\theta_{n,x}$, and write simply
%$\theta_x:\Omega\to\Omega$, which is defined as $(\theta_x\omega)(y):=\omega(x+y)$ for $y\in\mathbb Z^d$.
%In this case we will say that
% $\{\theta_{x}:x\in \mathbb Z^d\}$ is an ergodic
%family of transformations acting on the space $(\Omega,
%\mathcal B( \Omega), \mathbb P)$
%or an {\it ergodic family of transformations with respect to $\mathbb P$},
%if whenever $A\in\mathcal B( \Omega)$
%is such that $\theta^{-1}_{z} A=A$ for every $z\in\mathbb Z^d$, then $\mathbb P(A)$ is $0$ or $1$. 

%Throughout, it will be usefull to introduce the
%random walk $Y_n:=(X_n,n)$ for $n\ge 0$.
Define the space-time {\it environmental process} as the discrete time Markov chain

\[
\bar\omega_n:=\theta_{n,X_n}\omega, \qquad n\ge 0,
\]
with state space $\Omega^{\mathbb N}$.  Here
the random walk $\{X_n:n\ge 0\}$ has law
$P_{0,0,\omega}$. 
In general, if $\omega$ is distributed according to some
law $\mu$ on $\Omega$, we define $P_\mu:=\int P_{0,0,\omega} d\mu$.
We will say that $\mu$ is an invariant distribution for the
environmental process if  $\bar\omega_n$ under $P_\mu$
%is independent of $n$ 
has identical distribution for $n\ge 0$.

Let $D\subset U$.
 We say that a  random environment with law $\P$ is {\it balanced
in  $D$} if for every $x\in\mathbb Z^d$
and $n\ge 0$
\[
\P\left(\sum_{e\in D}e\omega_n(x,e)=0\right)=1.
\]
We say that the environment is {\it uniformly elliptic in $D$} with ellipticity constant  $\kappa>0$ if
\[
\mathbb P\left(\omega_n(x,e)>\kappa\text{ for all $e\in D\setminus\{0\}$, $x\in\Z^d$ and $n\ge 0$}\right)=1.
\]
When the environment is balanced (resp. uniformly elliptic) in $U$,
we simply say that it is balanced (resp. uniformly elliptic). We call the environment {\it elliptic}  if the jump range $U$ spans $\R^d$. 
%When $\kappa=0$ we will say that the
%environment is {\it elliptic}, while if
%$\kappa>0$ that it is {\it uniformly elliptic}. 

Let us now recursively define the range
of the random walk after $n$ steps as
$U_1:=U$, while for $n\ge 1$
\[
U_{n+1}:=\left\{y\in\mathbb Z^d:y=x+e\ {\rm for}\ {\rm some}\
x\in U_n\ {\rm and}\ e\in U\right\}.
\]
Let also for $n\ge 1$ and $x,y\in \mathbb Z^d$
\[
p_{n}(x,y):=P_{x,0,\omega}(X_n=y).
\]
Now set

\begin{equation}
\label{ven}
V_{n}(x):=\{p_n(x,x+z)z :z\in U_n\} \subset\R^d
\end{equation}
Throughout, given a subset $V\subset\mathbb R^d$, we will
denote by $conv(V)$ its convex hull and by
$|V|$ its Lebesgue measure.
Define for $n\ge 1$

\begin{equation}
\label{varen}
\varepsilon_n(x):=\left(p_n(x,x)\left|conv(V_n(x))\right|\right)^{1/(d+1)}.
\end{equation}
Denote also by $\{e_1,\ldots,e_d\}$ the canonical  basis
of $\mathbb Z^d$.

We say that the random walk $X_\cdot$ in random environment satisfies the {\it quenched central limit theorem} (QCLT) with a non-degenerate covariance matrix $A$ if
\begin{quote}
\it For almost all environments $\omega$, under $P_{0,0,\omega}$, the sequence $X_{[n\cdot]}/\sqrt{n}$ converges
in law to a Brownian motion with a deterministic non-degenerate covariance matrix $A$.
\end{quote}

\begin{theorem}
\label{qclt} Consider a discrete time  random walk
in an elliptic balanced time-dependent random environment with law $\mathbb P$. Suppose that the
 family of shifts $\{\theta_{1,e}:e\in U\}$
is ergodic and that

\begin{equation}
\label{ellipt-cond}
\inf_{n\ge 1}\mathbb E_{\mathbb P}\left[\varepsilon_n^{-(d+1)}(0)\right]<\infty.
\end{equation}
 Then, the following are satisfied.

\begin{itemize}

\item[(i)] The environmental process has a unique invariant 
probability measure $\nu$
which is absolutely continuous with respect to $\P$.

\item[(ii)] 
%\sout{$\mathbb P$-a.s., under $P_{0,0,\omega}$,
%the sequence $X_{[n\cdot]}/\sqrt{n}$ converges
%in law to a Brownian motion with a deterministic}
The QCLT holds with a
 non-degenerate covariance
matrix $A=\{a_{i,j}:1\le i,j\le d\}$, where

\[
a_{i,j}:=\sum_{e\in U}(e\cdot e_i)(e\cdot e_j)\int \omega_0(0,e)d\nu.
\]
\end{itemize}

\end{theorem}

%%%GGG<:

Theorem \ref{qclt} extends the static version of the QCLT proved by Lawler \cite{L-82} for
uniformly elliptic environments and by Guo and
Zeitouni \cite{GZ-10} for elliptic environments.
Other recent related results for random walks
in balanced static environments include
 Berger and Deuschel  \cite{BD-14}
and Baur \cite{Ba-14}. On the other hand,
it should be pointed out that several results
exist proving QCLT
for random walks in time-dependent environments,
but in general under mixing condition which are
stronger that our ergodicity assumption
(see for example \cite{DKL-08} or \cite{A-14}).
Recently in \cite{ACDS-16}, the QCLT is obtained for continuous-time random walk in time-dependent ergodic random conductance under similar moment conditions as in \eqref{ellipt-cond-cont}  on the jump rates.

\begin{rmk}
For nearest-neighbor random walks in a static random environment, a similar criteria for QCLT as \eqref{ellipt-cond} for $n=1$ is obtained by Guo and Zeitouni \cite{GZ-10}, 
where $\mb P$ is ergodic with respect to the spatial shifts $\{\theta_x: x\in\Z^d\}$. 
Note that in the time-dependent case, neither do we demand the environment to be ergodic under the spatial shifts, nor under the time shifts alone.

One may replace the ergodic family of transformations $\{\theta_{1,e}, e\in U\}$ with $\{\theta_{n,x}: (n,x)\in Z\}$ for some set $Z\subset\N\times\Z^d$. Clearly, the smaller the set $Z$ is, the stronger our ergodicity assumption is. A natural question is, can we weaken our condition by enlarging the ergodic family of transformations? We give a negative answer through a counterexample in Section~\ref{section-counterexample}, where the ergodic family of transformations is larger than $\{\theta_{1,e},e\in U\}$ but the QCLT fails. (In that case, a CLT holds for almost all $\omega$ but with a random covariance matrix dependent on $\omega$.)
\end{rmk}
%%%GGG>

\begin{rmk}  The non-degeneracy of the matrix $A$ of part $(ii)$ of
 Theorem \ref{qclt}
follows from the fact that for any vector $u=(u_1,\ldots,u_d)$
one has that
\[
u\cdot Au=\sum_{e\in U}(e\cdot u)^2\int \omega_0(0,e)d\nu>0.
\]
Indeed, \eqref{ellipt-cond} implies that
 the vectors of $U$ span $\mathbb R^d$. 
On the other hand,
$\nu$ is absolutely continuous with respect to $\P$, which implies that
$A$ is a positive-definite matrix.
\end{rmk}

\begin{rmk} Condition \eqref{ellipt-cond} of Theorem~\ref{qclt}
is always satisfied if the environment is balanced and uniformly elliptic  (with constant $\kappa>0$) in a subset $D\subset U$ such that $\left|conv(D)\right|>0$.
Indeed, since the environment is balanced in D, we see that
$0\in conv(D)$, so that for some constants $\lambda_e>0$, $e\in D$, we have $\sum_{e\in D}\lambda_e e=0$.  Moreover,  
the coefficients $\lambda_e$ can always be chosen as integer numbers, as all $e\in U$ have integer coordinates.
Therefore, the random walk returns
to the origin after $N=\sum_{e\in D}\lambda_e$ steps
with a probability larger than $\kappa^N$, so that
$p_N(0)\ge\kappa^N$. On the other hand, the fact
that $\left|conv(D)\right|>0$ implies that $|conv(V_N)|>0$,
c.f. \eqref{ven}, is also bounded by some positive
constant so that $\varepsilon_N$, c.f. (\ref{varen}), is bounded from below
by some positive constant.
\end{rmk}

\subsection{Continuous time RWRE}
We can also formulate a continuous time version of Theorem \ref{qclt}.
%%%GGG<: 
%Denote by $G:=\{e_1,e_{-1},\ldots,$ $e_d,e_{-d}\}$ the set of unit vectors
%in $\mathbb Z^d$.
Recall that $U$ is a finite subset of $\mathbb Z^d$. 
%%%GGG>
Define 
\[
\mathcal Q:=\big\{v=\{v(e)>0:e\in U\}\big\}.
\]
Note that we do not assume any upper bound on the transition rates $v(\cdot)\in \mathcal Q$. We call $D([0,\infty);\mathcal {\mathfrak H})$,
where  ${\mathfrak H}:=\mathcal Q^{\mathbb Z^d}$,
the {\it environmental space} while an element
\[
\omega:=\{\omega_t:t\ge 0\}
 \in D([0,\infty);\mathcal {\mathfrak H})
 \]
a {\it continuous time environment},
 so that $\omega_t:=
\{\omega_t(x):x\in\mathbb Z^d\}$ with $\omega_t(x):=\{\omega_t(x,e):e\in U\}
\in\mathcal Q$. 
Let us denote by $\mathbb Q$ the law of the continuous time environment $\omega$. 
Given an environment $\omega$, for $u:\mathbb Z^d\times [0,\infty)\to
\mathbb R$ bounded and differentiable in time for each $x\in\mathbb Z^d$, we define the parabolic difference operator 
%%%GGG:
\[
\mathcal L_\o u(x,t):=\sum_{e\in U}\o_t(x,e)[u(x+e,t)-u(x,t)]+\partial_t u(x,t).
\]
Let $(X_t,t)_{t\ge 0}$  be the Markov process on $\Z^d\times[0,\infty)$ with generator $\mathcal L_\o$ and initial state $(0,0)$.
%consider the continuous time random walk defined by
%the generator
%{\red It is better to define the generator for the space-time process.}
%$$
%L_sf(x):=\sum_{e\in U}\omega_s(x,e)(f(x+e)-f(x)),
%$$
%where $s\ge 0$ and $f:\mathbb Z^d\to\mathbb Z^d$ is bounded.
We call $(X_t)_{t\ge 0}$ a {\it continuous time random walk in 
the time-dependent environment $\o$} 
 and denote for each
$x\in\mathbb Z^d$ by $P^c_{x,t,\omega}$ the law on
$D([0,\infty);\mathbb Z^d)$ of this random walk starting from $x$
at time $t$.
We call $P^c_{x,t,\omega}$ the {\it quenched law}  of the
random walk. 

%\sout{For the continuous time environment, we will make the
%following ergodicity assumption.} 
For each $s\ge 0$ and $x\in\mathbb Z^d$
define the transformation 
\[
\theta_{s,x}:D([0,\infty);  \mathfrak H)
\to D([0,\infty); \mathfrak H)
\] 
by $(\theta_{s,x}\omega)_t(y):=\omega_{t+s}(x+y)$. We assume that the law of the environment
$\mathbb Q$ is stationary under the action of the shifts
 $\{\theta_{s,x}:s\ge 0,x\in \mathbb Z^d\}$.
%We say that the shifts $\{\theta_{s,x}:s>0,x\in  U\}$
%is an {\it ergodic family of transformations} for $\mathbb Q$ if whenever $A\in\mathcal 
%B(D([0,\infty);\mathfrak H))$ is such that $\theta^{-1}_{s,x}A=A$
%for every $s>0$ and $x\in  U$, then $\mathbb Q(A)$ is
%$0$ or $1$. 

We say that the  random environment $\omega$ with law $\mathbb Q$ is {\it balanced}
if for every $t\ge 0$, $x\in\mathbb Z^d$,
\begin{equation}
\label{bcte}
\sum_{e\in U}e\omega_t(x,e)=0
\qquad\mathbb Q-a.s.
\end{equation}
As in the discrete-time case,  we can also define the {\it environmental process}
as the continuous time Markov process
\[
\bar\omega_t:=\theta_{t,X_t}\omega
\]
for $t\ge 0$.
Here the process $(X_t)_{t\ge 0}$ is sampled according to $P^c_{0,0,\o}$.
%, with state space $D([0,\infty);\mathfrak H)$
%and initial state $\omega\in D([0,\infty);\mathfrak H)$. 
%Here
%the random walk $\{X_t:t\ge 0\}$ has the law $P^c_{0,0,\omega}$. 
%We will assume that
%it also has state space $\mathfrak H$  so that it is
%a stochastic process defined in $D([0,\infty);\mathfrak H)$
% with starting point $\omega \in D([0,\infty);\mathfrak H)$.  
 In general, if $\omega$ is distributed according to some 
law $\mu$, we define  $P^c_\mu:=\int P^c_{0,0,\omega} d\mu$. We will say
that $\mu$ is an invariant distribution for the environmental
process if the law of $\bar\omega_t$ under $P^c_\mu$ is independent
of $t$ for $t\ge 0$.

%%%GGG<:
For each $(x,t)\in\mathbb Z^d\times[0,\infty)$, let
\begin{equation}
\label{convexu}
U_{x,t}=\{\omega_t(x,e)e:e\in U\}
\end{equation}
 and
\begin{align}\label{epsilon-cont}
&\varepsilon(x,t)=\varepsilon_\omega(x,t):=|conv(U_{x,t})|^{1/(d+1)},\\
&\upsilon(x,t)=\upsilon_\omega(x,t):=\sum_{e\in U}\omega_t(x,e).\nonumber
\end{align}
%\begin{eqnarray*}
%&\varepsilon (x,t)=
%\varepsilon_\omega(x,t):=\left(\prod_{e\in G} \omega_t(x,e)
%\right)^{1/(d+1)}
%\qquad{\rm and}\\
%& \upsilon(x,t)=\upsilon_\omega(x,t):=\sum_{e\in G}\omega_t(x,e).
%\end{eqnarray*}
%%%GGG>
We will denote by $\mathbb E_{\mathbb Q}$ the expectation with respect
to the law $\mathbb Q$ of the environment and write $\varepsilon=
\varepsilon(0,0)$ and $\upsilon=\upsilon(0,0)$.

\begin{theorem}
\label{qclt-continuous} Consider a continuous time random walk in an elliptic
time-dependent balanced random environment with  law $\mathbb Q$. Suppose that the
family of shifts $\{\theta_{s,x}:s>0,x\in U\}$
is ergodic. Assume that

%%%GGG<:
\begin{equation}\label{ellipt-cond-cont}
\mb E_{\mathbb Q}\left[\frac{\upsilon^{d+1}+1}{\varepsilon^{d+1}}\right]
<\infty,
\end{equation}
%%%GGG>
Then, the following are satisfied.

\begin{itemize}

\item[(i)] The environmental process has a  unique invariant distribution
which is absolutely continuous with respect to $\mathbb Q$.

\item[(ii)] $\mathbb Q$-a.s. under $P^c_{0,0,\omega}$
the sequence $X_{t\cdot}/\sqrt{t}$ converges, as $t\to\infty$,
in law on the Skorokhod space $D([0,\infty);\mathbb R^d)$ to a Brownian motion with a deterministic non-degenerate  covariance matrix.

\end{itemize}
\end{theorem}

\subsection{Applications of the QCLTs}

Theorem~\ref{qclt-continuous}  can be
applied to derive quenched central limit theorems
for balanced environments driven by
some interacting particle systems.
 An example of this situation is a random
walk moving among the zero-range process. 
Given a function $g:\mathbb N\to\mathbb [0,\infty)$
satisfying 
 $g(k)>g(0)=0$ for
all $k>0$, the zero-range process
can be constructed as a
Markov process describing the movement of particles
on the lattice $\mathbb Z^d$, so that
if at a site $x\in\mathbb Z^d$ and time
$t\ge 0$ there are $\eta_t(x)$ particles,
a particle jumps uniformly to the nearest
neighboring sites of $x$ at a rate $g(\eta_t(x))$.
The infinitesimal generator $L$ of this interacting
particle system is defined by its action
on functions $f:\mathbb N^{\mathbb Z^d}\to\mathbb R$
depending on a finite number of coordinates
of $\eta=\{\eta(x):x\in\mathbb Z^d\}\in\mathbb N^{\mathbb Z^d}$
by
\[
Lf(\eta)=\sum_{x,y\in\mathbb Z^d:|x-y|_1=1}g(\eta(x))(f(\eta^{x,y})-f(\eta)),
\]
where
 \[
\eta^{x,y}(z):=\eta(z)-1_{z=x}+1_{z=y}.
\]
%$$
%\eta^{x,y}(z):=\begin{cases}
%\eta(x)-1&\quad\rm{if}\quad z=x\\
%\eta(y)+1&\quad\rm{if}\quad z=y\\
%\eta(z)&\rm{if}\quad z\ne x,y
%\end{cases}
%$$
Under the condition
\[
\sup_{k\in\mathbb N}|g(k+1)-g(k)|<\infty,
\]
this process is well defined whenever the
initial condition $\eta\in S$, where
\[
%\label{state-space}
S:=\{\eta\in{\mathbb N}^{\mathbb Z^d}:\sum_{x\in\mathbb Z^d}\eta(x)\alpha(x)<\infty\},
\]
and
\[
\alpha(x):=\sum_{n=0}^\infty\frac{1}{2^n}p_n(0,x),
\]
with $p_n(0,x)$ the probability that a discrete time simple symmetric
random walk starting from $0$ is at position $x$ at time $n$
(see \cite{A-82} for this construction).
The above process is called {\it zero-range process},
and we will denote by $P_\eta$ its law
on $D([0,\infty);S)$ starting from $\eta\in S$. This process
has a family of invariant measures (see also
\cite{A-82}) defined through
the partition function $Z:[0,\infty)\to [0,\infty)$ by
\[
Z(\alpha)=\sum_{k\ge 0}\frac{\alpha^k}{g(1)\cdots g(k)}.
\]
Define
\[
\alpha^*:=\sup\{\alpha\in [0,\infty):Z(\alpha)<\infty\}.
\]
Assume also that
\[
\lim_{\alpha\to\alpha^*}Z(\alpha)=\infty.
\]
Now, for each $0\le\alpha<\alpha^*$ define the product
probability measure $\mu_\alpha$ on the Borel $\sigma$-algebra of $\mathbb N^{\mathbb Z^d}$, with marginals given by
\[
\mu_\alpha(k)=\frac{1}{Z(\alpha)}\frac{\alpha^k}{g(1)\cdots g(k)}.
\]
As a matter of fact $\mu_\alpha(S)=1$, so that we can define
a probability measure
\[
P_{\alpha}:=\int P_\eta d\mu_\alpha(\eta).
\]
Let us assume that the function $g$ is non-decreasing, 
 so that for each $0\le\alpha<\alpha^*$,
%. In this
%case, for each $\alpha\in [0,\alpha^*)\}$,
the invariant measure $\mu_\alpha$ is also extremal
\cite{A-82}. Therefore, for
 $\alpha\in [0,\alpha^*)$,
 the
shifts $\{\theta_{s,x}:s>0,x\in\mathbb Z^d\}$ form an
ergodic family of transformations for $P_{\alpha}$.
For each $e\in \{e_1,\ldots, e_d\}$, fix a finite range 
function $u(e,\cdot):S\to\mathcal (0,\infty)$:
in other words, there is an $R>0$ such that for
each $\eta\in S$, $u(e,\eta)$ depends only
on $\eta(x)$ for $|x|_1\le R$. Define
for $e\in\{e_1,\ldots,e_d\}$, 
\[
u(-e,\cdot):=u(e,\cdot)
\]
Now, define the
stochastic process for $t\ge 0$,
 $\omega_t:=\{\omega_t(x):x\in\mathbb Z^d\}$
where for $x\in\mathbb Z^d$, we define
\[
\omega_t(x):=\{ u(e,\theta_x\eta_t):e\ \rm{such}\ \rm{that}\ |e|_1=1\}.
\]
Let us call $\omega:=\{\omega_t:t\ge 0\}$ a {\it an
environment generated by an attractive zero-range process}.
We then have the following immediate corollary to Theorem~\ref{qclt-continuous}.

\begin{corollary}
\label{coro1} Consider a continuous time random walk $\{X_t:t\ge 0\}$ in
an environment $\omega$ generated by an attractive zero-range
process, with law $P_\alpha$ for some $\alpha\in [0,\alpha^*)$.
Assume that

\begin{equation}
\label{cond-coro}
\int\frac{\left(\sum_{e\in U} u( e,\eta)\right)^{d+1}+1}
{\prod_{i=1}^du(e_i,\eta)}
d\mu_\alpha<\infty.
\end{equation}
Then $P_\alpha$-a.s. the random walk $\{X_t:t\ge 0\}$ satisfies
the quenched central limit theorem with nondegenerate covariance
matrix.
\end{corollary}

 The condition \eqref{cond-coro} is satisfied
when the function $u$ is bounded from both above and below.
Furthermore, Corollary~\ref{coro1} includes the case of a second
class particle in the zero-range process, solved by Saada in \cite{Sa-90},
where nevertheless the main problem we face here, which
is the construction of the invariant measure, is
not present.

 Theorem~\ref{qclt} can be applied to derive QCLTs for a certain class of random walks in static random
environment. In order to give a simple example, we will consider a random walk on $\mathbb Z^{d_1+d_2}$, $d_1,d_2\in\N$. For
$x\in\mathbb Z^{d_1+d_2}$, we write 
$x=(x^{(1)},x^{(2)})$ so that $x^{(1)}\in\Z^{d_1}$ and $x^{(2)}\in\Z^{d_2}$.
We say that a static environment $\omega\in\mathcal P^{\Z^{d_1+d_2}}$  %Consider then a random walk $X_n=(X_n^{(1)}, X_n^{(2)})$ on $\mathbb Z^{d_1+d_2}$ in  a static random environment with jump range
% $U=\{e\in\Z^{d_1+d_2}: |e|_1=1\}$. We say that an environment $\omega\in\Omega$ 
is {\it autonomous in the first coordinates} if
$P_{x,\omega}(X_1^{(1)}=e^{(1)})$ depends only on the first $d_1$ coordinates of $x$. That is, for $x,z\in\Z^{d_1+d_2}$,
\[
P_{x,\omega}(X_1^{(1)}=e^{(1)})=P_{z,\omega}(X_1^{(1)}=e^{(1)}) \quad\text{ if $x^{(1)}=z^{(1)}$.}
\]
In other words, $X_n^{(1)}$ is a Markov chain under $P_{0,\omega}$.

\begin{corollary}
\label{qclt-static}
Consider a random walk $X_n=(X_n^{(1)}, X_n^{(2)})$ on $\mathbb Z^{d_1+d_2}$
in a static random environment  with jump range $U=\{e\in\Z^{d_1+d_2}: |e|\le 1\}$.
 Assume that $\P$ is stationary under $\{\theta_x: x\in\Z^{d_1+d_2}\}$ and
for all $i=1,\ldots, d_1$, it is ergodic under the shifts 
$\{\theta_{\pm e_i}\}$. Suppose the following are satisfied.

\begin{enumerate}[(a)]

\item {\bf (autonomous first coordinates)} 
$\P$-almost surely, the environment is autonomous in the first coordinates.

\item {\bf (equivalent ergodic invariant
measure for the first coordinates)} There is an invariant
measure $\nu$ which is equivalent to $\P$ such that the
environmental process $(\theta_{(X_n^{(1)},0)}\omega)_{n\in\N}$ is an ergodic sequence under $\nu\times P_{0,\omega}$.

\item {\bf (QCLT for the first coordinates of the random walk)} There exists
a deterministic vector $v_1\in\Z^{d_1}$ such that 
$\P$-a.s. under $P_{0,\omega}$, the sequence $(X_{[n\cdot ]}^{(1)}-v_1n\cdot)/\sqrt{n}$ converges in law to a Brownian motion with a deterministic non-degenerate covariance matrix.

\item {\bf (balanced and uniformly-elliptic in the second coordinates)} 
$\P$-a.s., $\omega$ is uniformly-elliptic and balanced in $\{\pm e_i: i=d_1+1,\ldots, d_1+d_2\}$.

%\item[(f)] {\bf (the second coordinate of the
%random walk can stay put)} 
%The random walk $\{Z_n:n\ge 0\}$ has at each step a positive
%probability of not moving.
\end{enumerate}
Then 
$\P$-a.s. under $P_{0,\omega}$, the sequence $(X_{[n\cdot ]}-nv)/\sqrt{n}$ converges in law to a Brownian motion with a deterministic non-degenerate covariance matrix,  where $v:=(v_1,0)\in\Z^{d_1+d_2}$.
\end{corollary}

\begin{rmk}
As an application of Corollary~\ref{qclt-static}, we will obtain QCLT for an environment which is ``ballistic and autonomous in the first $d_1$ directions and balanced in the other $d_2$ coordinates".

We let $\alpha\in\mathcal P^{\Z^{d_1}}$ and $\beta\in\mathcal P^{\Z^{d_1+d_2}}$ be static random environments which are independent under their joint law. We construct $\alpha, \beta$ such that
\begin{enumerate}[(i)]
\item 
$\alpha$ is an iid uniformly-elliptic environment on $\Z^{d_1}$, $d_1\ge 4$, with jump range $\{y\in\Z^{d_1}:|y|\le 1\}$ and it satisfies the ballisticity condition $(\mathscr P)$, c.f. \cite[Theorem~1.10]{BCR-14}.
\item 
$\beta$ is a stationary (under the shifts $\{\theta_x: x\in\Z^{d_1+d_2}\}$) balanced environment on $\Z^{d_1+d_2}$ which is uniformly-elliptic on  its jump range $\{x\in \Z^{d_1+d_2}: x^{(1)}=0, |x^{(2)}|\le 1\}$.
\end{enumerate}
Define the environment $\omega\in\mathcal P^{\Z^{d_1+d_2}}$ such that for $x,e\in\Z^{d_1+d_2}$,
\[
\omega(x,e)=
\left\{
\begin{array}{rl}
\alpha(x^{(1)}, e^{(1)}) &\text{ if }e=\pm e_i,\, i=1,\ldots,d_1\\
\alpha(x^{(1)},0)\beta(x,e) &\text{ if }e=0\text{ or }e=\pm e_{d_1+j},\, j=1,\ldots d_2.
\end{array}
\right.
\]
Then the law of $\omega$ satisfies all conditions of Corollary~\ref{qclt-static} and the QCLT holds.
\end{rmk}

\begin{rmk}
One may replace $\alpha$ in the above example with any environment that satisfies conditions (b) and (c) of Corollary~\ref{qclt-static}. 
For instance, let $\xi(y,e)$ be the transition probability of an iid conductance model, we may take a constant $r\in(0,1)$ and let $\alpha(x,e):=(1-r)\xi(x,e)+r\mathbb 1_{e=0}$. Clearly, the presence of $r$ is to guarantee that $\alpha(x,0)>0$.
\end{rmk}

\begin{rmk}
Recently Baur \cite{Ba-14} proved a QCLT with similar flavor for iid static environment on $\Z^{d}, d\ge 3$, where the environment is a small perturbation of the simple random walk and is balanced in a fixed coordinate direction. The law of the environment is also assumed to be invariant under antipodal reflections.
\end{rmk}

\subsection{Organization of the article}

 Since the proofs of Theorem~\ref{qclt} and \ref{qclt-continuous}
are similar, most of the subsequent sections of this paper
will give the details of the proof of the discrete time
Theorem~\ref{qclt}, while an outline of the proof of
the continuous time Theorem~\ref{qclt-continuous}
is given in section~\ref{section-continuous}.  In section~\ref{section-kozlov}, we state the version of Kozlov's theorem
that will be used to construct the absolutely continuous invariant
measure for the discrete time random walk. In subsection~\ref{section-parabolic},
the parabolic maximum principle for general meshes is stated
while its proof is deferred to subsection~\ref{proof-maximal}. Both Kozlov's theorem and the parabolic maximum principle
are subsequently used in section~\ref{proof-main} to prove
Theorem~\ref{qclt}. Corollary~\ref{qclt-static} is proved
in section~\ref{section-application}. In section~\ref{section-counterexample},
we give an example that the ergodicity hypothesis of Theorem~\ref{qclt}
cannot be weakened by enlarging the ergodic family of transformations.
%cannot be changed to a weaker kind of total ergodicity.

\section{Two preliminary tools}
Here we state two theorems that will be used to prove
Theorem~\ref{qclt}. The first is a version of a
well known theorem of Kozlov  for time dependent random walks, while the second is
the parabolic maximum principle for general meshes,
whose proof is given in Section~\ref{proof-maximal}. The parabolic
maximum principle is a crucial tool to construct
the absolutely continuous invariant measure of part $(i)$
of Theorem~\ref{qclt}, while Kozlov's theorem is required
to derive part $(ii)$ of the same theorem.

\subsection{Kozlov's theorem}
\label{section-kozlov}
The proof of {Theorem~\ref{qclt}} will require
the following version of Kozlov's theorem \cite{Ko-85}
for time dependent random walks.

\begin{theorem}
\label{kozlov} Consider a
random walk in
a time-dependent  elliptic  random environment 
which has a law $\mathbb P$. Assume
that $\{\theta_{1,z}:z\in  U\}$ is an ergodic
family of transformations with respect to $\mathbb P$.
 Assume that there exists an invariant measure $\nu$
for the environmental process, which is absolutely
continuous with respect to $\mathbb P$.
Then, the following are satisfied:

\begin{itemize}

\item[(i)] $\nu$ is equivalent to $\mathbb P$.

\item[(ii)] The environment as seen from the random walk
with initial law $\nu$ is ergodic.

\item[(iii)] $\nu$ is the unique probability measure for the environmental
process which is absolutely continuous with respect to $\mathbb P$.

\end{itemize}
\end{theorem}

\begin{proof} Since the proof is similar to the case of
random walks in static random environments, we will
stress the steps which are different  (see Theorem 1.2 of Lecture 1 of
\cite{BS-02},
for example, for a proof of the theorem for static random walk
in random environment). 

To prove part $(i)$, let $f$ be the  Radon-Nikodym derivative of $\nu$ with
respect to $\mathbb P$ and define $E:=\{f=0\}$. We will prove that
$\mathbb P(E)=0$. Using the fact that $\nu$ is invariant, 
we can conclude that $\mathbb P$-a.s.
for every $z\in U$,
\[
1_E(\omega)\ge \sum_{z'\in U}\omega_0(0,z')1_E(\theta_{1,z'}\omega)
\ge\omega_0(0,z)1_E(\theta_{1,z}\omega).
\]
From the ellipticity assumption and the fact that $1_E(\omega)$
and $1_E(\theta_{1,z}\omega)$ only take the values $0$
and $1$ we see that for each $z\in U$, $\mathbb P$-a.s.
\[
1_E(\omega)\ge 1_E(\theta_{1,z}\omega).
\]
Now since $\mathbb P(E)=\mathbb P(\theta_{1,z}^{-1}E)$, we conclude
that for each $z\in U$, $\mathbb P$-a.s.

\[
%\label{e-inv}
1_E= 1_{\theta_{1,z}^{-1}E}.
\]
%By a similar reasoning we conclude that for every $n\ge 0$ and $z\in U_n$,
%
%$$
%1_E= 1_{\theta_{n,z}^{-1}E}.
%$$
%Now define
%
%$$
%\tilde E:=\cap_{m=0}^\infty \cup_{k=m}^\infty \cup_{y\in U_k}\theta_{k,y}^{-1}E.
%$$
%Note that for every $z\in U$, we have that
%
%$$
%\tilde E=\theta_{1,z}^{-1}\tilde E.
%$$
%Since $\mathbb P$ is ergodic with respect to the
%family of space-time shifts $\{\theta_{1,z}:z\in U\}$,
%and $\tilde E$ differs from $E$ on a set of measure $0$,
Thus the function $1_E:\Omega^\N\to\{0,1\}$ is a.-s. shift-invariant under $\{\theta_{1,z}:z\in U\}$. By our ergodicity assumption, we conclude that $1_E$ is a-s. a constant and so
\[
\mathbb P( E)
%=\mathbb P(\tilde E)
\in\{0,1\}.
\]
But $\int_{E^c}fd\mathbb P=\int fd\mathbb P=1$ implies that
$\mathbb P(E^c)>0$, so that necessarily $\mathbb P(E)=0$.

Let $T:\Omega^{\mathbb N\times\N}\to\Omega^{\mathbb N\times\N}$ denote the left-shift that maps the sequence $(\bar\omega_n)_{n\ge 0}$ to $(\bar\omega_{n+1})_{n\ge 0}$.
To prove part $(ii)$ as in the static case (see \cite{BS-02})
it is possible to prove that for every $A\in\mathcal B((\Omega^{\mathbb N})^\N)$
such that $T^{-1}A=A$, 
 the process
%%%GGG<
\[
\phi(\bar\omega_n):=P_{\bar\omega_n}(A),
\]
%%%GGG>
is a martingale and also there is a set $B\in\mathcal B(\Omega^\N)$
such that
\[
\phi=1_B.
\]
We then show that $\mathbb P$-a.s. for each $z\in U$, the inequality
\[
1_B(\omega)\ge\sum_{z'\in U}\omega_0(0,z')1_B(\theta_{1,z'}\omega)
\ge \omega_0(0,z)1_B(\theta_{1,z}\omega)
\]
is satisfied. Using an argument similar to the
one employed in part $(i)$ we now see that $\nu(B)\in\{0,1\}$,
which proves that $P_\nu(A)=\nu(B)\in\{0,1\}$.

The uniqueness of $\nu$ stated in part $(iii)$ can be obtained following
exactly the same argument as in the static case \cite{BS-02}.
\end{proof}

\subsection{A new  maximum principle for parabolic difference operators on general meshes}
\label{section-parabolic}
The quenched central limit theorem for random walks in static balanced random environments \cite{L-82}
can be proved using lattice versions of the
maximum principle for elliptic operators
of Aleksandrov-Bakel'man-Pucci \cite{A-63,B-61,P-66} for elliptic partial
differential equations 
(see Papanicolaou and Varadhan \cite{PV-82}
for an application to prove a QCLT for diffusions with random
coefficients).
The maximum principle for elliptic difference operators were proved by Kuo and Trudinger in a series of papers
(see for example \cite{KT-90}).

Nevertheless, to prove Theorem~\ref{qclt}, we will need a parabolic maximum principle. 
 Within the context of diffusions,
this was first established by Krylov 
\cite{Kr-76}, and subsequently a discrete version for general meshes proved
by Kuo and Trudinger in \cite{KT-93,KT-95,KT-98}. Here
we state  a new  parabolic maximum principle, Theorem~\ref{maximum}, for difference operators and
prove it in section~\ref{proof-maximal} using a geometric approach.

We firstly introduce some notation. Given $x\in\mathbb Z^d$,
we denote by $|x|_2$ its $l_2$ norm. 
%%%GGG: It seems the notation $x\sim y$ is not used any longer, so I comment it out:
%Furthermore, for $x,y\in \mathbb Z^d$,
%we will write $x\sim y$ whenever $|x-y|_2=1$.
For $x_0\in\Z^d$, $R>0$, let 
\[
B_R(x_0):=\{x\in\Z^d: |x-x_0|_2\le R\},
\]
Consider a balanced time dependent environment 
$a=\{a_n:n\ge 0\}\in\Omega^{\mathbb N}$ 
(c.f. \eqref{environment}).
For any finite subset $\mathcal D\subset\Z^d\times\Z$, define its {\it parabolic
boundary} by
\[
\mathcal D^p:=\left\{ (y,n+1)\notin\mathcal D: \quad a_{n}(x,y-x)>0\text{ for  some } (x,n)\in\mathcal D\right\}.
\]
\begin{figure}[h]
\centering
\includegraphics[scale=0.40]{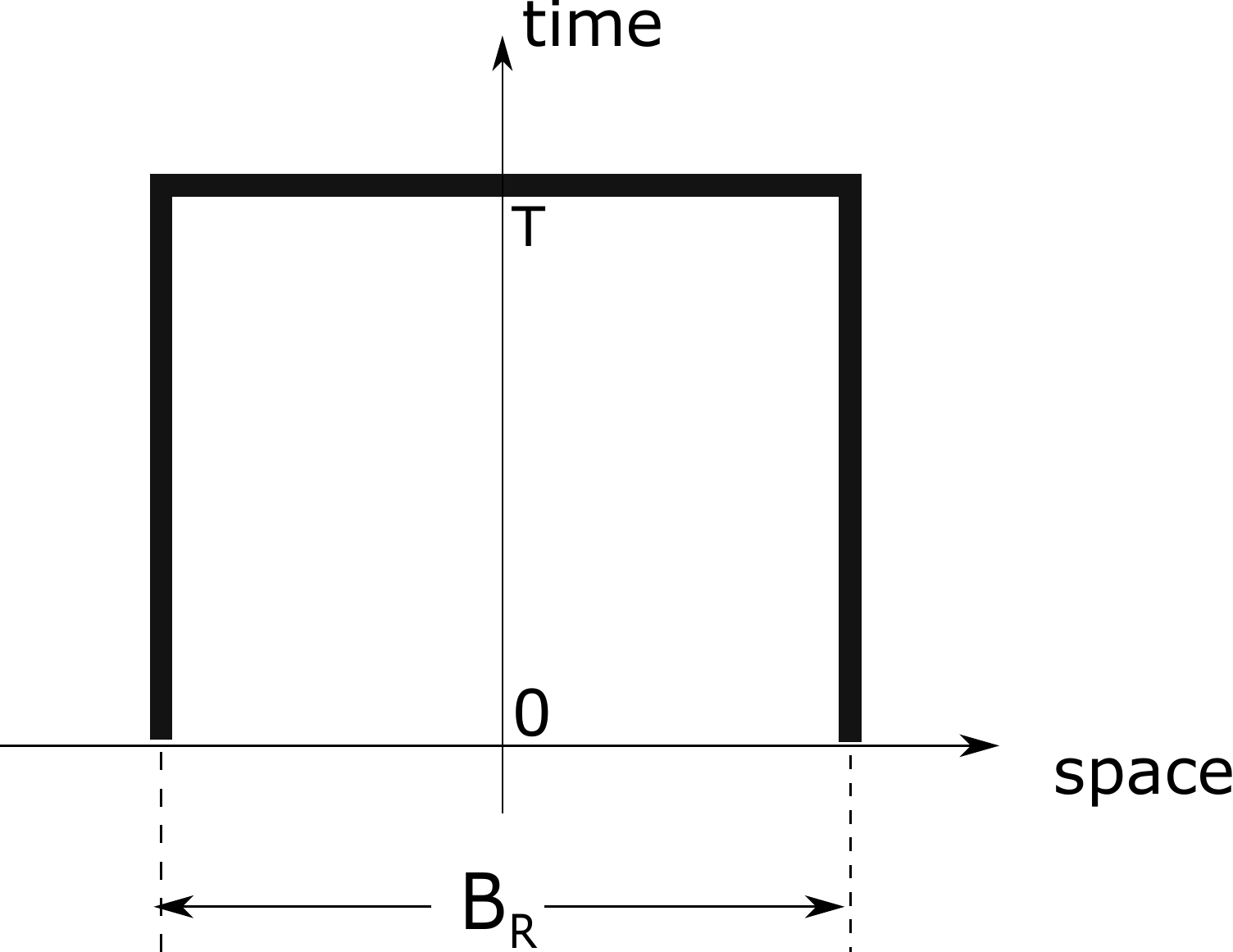}
\caption{The thick black lines represent the  parabolic boundary of $B_R\times[0,T)$.}
\end{figure}
%\footnote{Every open set $\mathscr D\subset\Z^d\times\R$ can be written as
%\[
%\mathscr D=\bigcup_{x\in D}\{x\}\times\mathcal O_x,
%\] 
%where $D$ is the $\Z^d$-projection of $\mathscr D$, and $\mathscr{O}_x$'s are open subsets of $\mathbb R$.}. 
 Define the parabolic operator
\[
\mathcal L_a u(x,n):=\sum_{z\in U}a_n(x,z)u(x+z,n+1)-u(x,n).
\]
For a real function $g$ defined on $\mathcal D\subset\Z^d\times\Z$
and $p>0$ define

\begin{equation}
\label{box-norm}
||g||_{\mathcal D,p}:=\left(\sum_{(x,n)\in\mathcal D}|g(x,n)|^p\right)^{1/p}.
\end{equation}
Set
\begin{equation}\label{def-u-v}
U_{x,n}:=\{a_n(x,z)z: z\in U\}
, \qquad
v(x,n):=\left|conv(U_{x,n})\right|,
\end{equation}
and define
\[
\varepsilon_a(x,n):=\left(a_n(x,0)\frac{v(x,n)}{\#U}\right)^{1/(d+1)},
\]
where $\# U$ denotes the cardinality of the discrete set $U\in\Z^d$.

\begin{theorem}
\label{maximum} Assume that $\mathcal D\subset \Z^d\times\Z$
is a finite set and $\mathcal D\cup\mathcal D^p\subset
B_R\times [0,T]$ for some $R,T>0$.
Let $u$ be a function on
$\mathcal D\cup\mathcal D^p$ that satisfies

\begin{equation}
\label{poisson-inequality}
\mathcal L_au\ge -f\qquad {\rm in}\quad \mathcal D
\end{equation}
for some function $f$ on $\mathcal D$. Then,  if $\varepsilon_a(x,n)>0$ for all $(x,n)\in\mathcal D$, we have
\[
\max_{\mathcal D}u\le\max_{\mathcal D^p}u
+CR^{d/(d+1)}
 \norm{f/\varepsilon_a}_{\mathcal D, d+1},
\]
 where $C=C(U,d)$ is a constant.

\end{theorem}

\begin{rmk}
The elliptic version of Theorem 2.2 was implicitly obtained in \cite[(44)]{KT-00}. However, there is a minor gap in its proof. That is, \cite[Lemma 3]{KT-00} is not true for general non-symmetric convex bodies. This can be fixed by symmetrization and using the balanced assumption (see \eqref{seven}).
\end{rmk}

\section{Proof of the discrete time QCLT (Theorem \ref{qclt})}
\label{proof-main}
It is easy to check, as in the case of
random walks in static balanced random environments, that
part $(i)$ of Theorem \ref{qclt} implies, through Theorem
\ref{kozlov}, part $(ii)$
(see \cite{L-82}). We therefore will concentrate on the proof of part $(i$).

Throughout this section, we fix a balanced environment $\omega\in\Omega^{\mathbb N}$,
so that for all $x\in\mathbb Z^d$ and $n\ge 0$,

\[
\sum_{e\in U}e\omega_n(x,e)=0
\]

By \eqref{ellipt-cond}, we know that there is $k\in\N$
such that the random walk returns to its starting
point after $k$ steps and such that

\begin{equation}
\label{finite-epsilon}
\mathbb E_{\mathbb P}\left[\varepsilon_k^{-(d+1)}\right]<\infty.
\end{equation}
We will soon see that the case in which $k>1$ can be
reduced to the case $k=1$. Therefore, let us first assume that
$k=1$. Let $N$ be an even natural number. We introduce for $(x,n)\in\mathbb Z^d\times\mathbb Z$ the equivalence classes
\begin{equation}
\label{eclass}
\overline{(x,n)}:=(x,n)+ (2N+1)\mathbb Z^d\times(N^2+1)\mathbb Z.
\end{equation}
In addition we define the
periodized version  $\omega^{(N)}$ of $\omega$ so that 
$\omega^{(N)}_{m}(y)=\omega_{n}(x)$
for $(y,m)\in\mathbb Z^d\times\mathbb Z$ 
with $\overline{(y,m)}=\overline{(x,n)}$ and
\[
(x,n) \in
K_N:=\{z\in\mathbb Z^d:|z|_\infty\le N\}\times\{ n':0\le n'\le N^2\}.
\]
Set
\[
\Omega_{N,\omega}=\Omega_N:=\{\theta_{n,x}\omega^{(N)}: (x,n)\in\mathbb Z^d\times\mathbb Z\}.
\]
It is straightforward to see that the process $(\theta_{n,X_n}\omega^{(N)})_{n\ge 0}$ is a Markov chain 
%{\red We should say that $X_n$ is a random walk under $\omega^{(N)}$. }
 with a finite state space $\Omega_N$ and has an invariant measure $\nu_N\ll \mb P_N$, where
\[
%%%GGG
\mathbb P_N:=\frac{1}{(N^2+1)}\frac{1}{(2N+1)^d}\sum_{(x,n)\in K_N}\delta_{\theta_{n,x}\omega^{(N)}}.
\]
Although it will not be used in this proof, note that
the  measure $\nu_N$ is of the form
\[
\nu_N=\sum_{(x,n)\in K_N}\phi_N(x,n)\delta_{\theta_{n,x}\omega^{(N)}},
\]
where $\phi_N=\phi_{N,\omega}$ is also the density of an invariant measure
of the random walk $\{\overline{(X_n,n)}:n\ge 0\}$, c.f. \eqref{eclass},
on $K_N$ in the environment $\omega^{(N)}$.

Note that since $\mathbb P$ is ergodic under the
action of $\{\theta_{1,x}:x\in U\}$ which is a subset of the transformations  $\{\theta_{n,x}:(n,x)\in\N\times\Z^d\}$, by the multidimensional ergodic theorem (see \cite[Theorem VIII.6.9]{DF-88}), 
\[
\lim_{N\to\infty}\mathbb P_N=\mathbb P\qquad\mathbb P-a.s.
\]

Define the stopping times $\tau_0=0$ and
\[
%\label{stopping}
\tau_{j+1}=\inf\{i>\tau_j: \abs{X_i-X_{\tau_j}}_\infty> N \mbox{ or } i-\tau_j>N^2\}, \quad j\ge 0.
\]

\begin{lemma}\label{L-one} There exists a constant $c_2>0$ such that for all
$c\ge c_2$, there is an $N_0$ such that for $N\ge N_0$  we have that
\begin{equation}
\label{lemma3}
\sup_{x\in\mathbb Z^d,n\ge 0,\xi\in\Omega_N}E_{x,n,\xi}\bigg[\left(1-\frac{c}{N^2}\right)^{\tau_1}\bigg]\le \frac{1}{2}.
\end{equation}
\end{lemma}
\pf
The proof follows the lines of \cite[Lemma 4]{GZ-10}.
Since $\{X_n:n\ge 0\}$ is a martingale, by Doob's martingale inequality, for any $1\le K<N^2$,
\begin{align*}
P_{x,n,\xi}(\tau_1\le K)
&\le 2\sum_{i=1}^d \sup_{\xi\in\Omega_N}P_{x,n,\xi}(\max_{0<m\le K}(X_{n+m}-x)^+(i)>N)\\
&\le \frac{2}{N}\sum_{i=1}^d \sup_{\xi\in\Omega_N}E_{x,n,\xi}[(X_{n+K}-x)^+(i)]\le \dfrac{2dC_U\sqrt K}{N},
\end{align*}
where for each $y\in\mathbb Z^d$, $y(i)$ denotes its $i$-th coordinate and $C_U=\max\{|e|:e\in U\}$. Hence for every $c>0$ we have that
\begin{align*}
E_{x,n,\xi}\bigg[\left(1-\dfrac{c}{N^2}\right)^{\tau_1}\bigg]
\le 
\left(1-\frac{c}{N^2}\right)^K+\frac{2dC_U\sqrt K}{N}.
\end{align*}
Taking $K=(\frac{N}{8dC_U})^2$ it follows that for $c$ large enough
whenever $N$ is large enough then inequality \eqref{lemma3} is satisfied.
\qed

Denote by $S=S_{\omega^{(N)}}$ the transition semigroup of the environment Markov chain $\{\xi_n:n\ge 0\}$
defined for each $\xi\in\Omega_N$ as $\xi_n:=\theta_{n,X_n}\xi$ for $n\ge 0$. That is, for every  function $g$ on $\Omega_N$ and  $\xi\in\Omega_N$, 
for $k\ge 0$ define
\[
S^kg(\xi):=E_{0,0,\xi} \left[g\left(\xi_{k}\right)\right], \quad\forall k\in\mathbb N.
\]
%%%GGG140820
Since $\nu_N$ is the invariant {distribution} for the finite-state Markov chain $(\xi_k)_{k\ge 0}$, we have 
\[
\int gd\nu_N=\int S^k gd\nu_N, \quad\forall k\in \mathbb N.
\] 
%%%GGG<
Let $c_2, N_0$ be the same constants as in Lemma~\ref{L-one}. For $N\ge N_0$, putting $\rho=\rho(\omega,N):=1-\frac{c_2}{N^2}\in(0,1)$, 
%%%GGG>
we see that
\begin{align}\label{eq1}
(1-\rho)^{-1}\int gd\nu_N 
&=\sum_{k=0}^\infty \rho^k\int S^k gd\nu_N\le \max_{\xi\in \Omega_N}\sum_{k=0}^\infty \rho^k S^kg(\xi)
=\max_{\xi\in\Omega_N} E_{0,0,\xi}
\bigg[\sum_{k=0}^\infty\rho^k g(\xi_k)\bigg].
\end{align}
On the other hand,  we have that
\begin{align}\label{eq2}
\max_{\xi\in\Omega_N} E_{0,0,\xi}
\bigg[\sum_{k=0}^\infty\rho^k g(\xi_k)\bigg]
&\le \sum_{m=0}^\infty \max_{\xi\in\Omega_N}E_{0,0,\xi}\bigg[\rho^{\tau_m} 
\sum_{k\in[\tau_m,\tau_{m+1})}g(\xi_k)
\bigg]\nonumber\\
&\le \sum_{m=0}^\infty \bigg(\max_{\xi\in \Omega_N}E_{0,0,\xi}[\rho^{\tau_1}]\bigg)^m
 \max_{\xi\in \Omega_N}E_{0,0,\xi}\bigg[\sum_{k=0}^{\tau_1-1}g(\xi_k)\bigg]\nonumber\\
&\le
2 
 \max_{\xi\in \Omega_N}E_{0,0,\xi}\bigg[\sum_{k=0}^{\tau_1-1}g(\xi_k)\bigg]
\end{align}
where in the first inequality we used the strong Markov property
at times $\tau_m, \tau_{m-1}$ up to $\tau_1$ successively,
and in the last inequality we used  inequality \eqref{lemma3} of Lemma~\ref{L-one}.
 Recall that $K_N^p$ denotes the parabolic boundary of $K_N$.
Now, for any $(x,n)\in K_N\cup K_N^p$ and $\xi\in\Omega_N$, 
define 
\[
f_\xi(x,n):=E_{ x,n,\xi}\left[\sum_{k=0}^{\tau-1}g(\xi_k)\right]
\]
where $\tau=\inf\{i\ge 0: |X_i|_\infty> N \mbox{ or }i> N^2\}$.
Then $f_\xi$ satisfies
\begin{equation*}
\left\{
\begin{array}{rl}
\mathcal L_af_\xi(x,n)=-G_\xi(x,n), & \text{if } (x,n)\in K_N\\
f_\xi(x,n)=0, \quad& \text{if } (x,n)\in K_N^p.
\end{array}
\right.
\end{equation*}
where $G_\xi(x,n):=g(\theta_{x,n}\xi)$.
We can now apply Theorem \ref{maximum}
to conclude 
\begin{align}\label{eq3}
\max_{\xi\in \Omega_N}f_\xi( 0,0)
&\le \max_{\xi\in\Omega_N} CN^{d/(d+1)}\norm{G_\xi/\varepsilon_1}_{K_N,d+1}\nonumber\\
&=CN^2\norm{g/\varepsilon_1}_{L^{d+1}(\P_N)},
\end{align}
where the norm $||\cdot ||_{K_N,d+1}$ is defined in \eqref{box-norm}.
Therefore, combining \eqref{eq1}, \eqref{eq2} and \eqref{eq3},
we conclude that for some constant $C>0$,
\[
%%%GGG:
\int gd\nu_N\le C \norm{g/\varepsilon_1}_{L^{d+1}(\P_N)}.
\]
Since $\Omega$ is pre-compact, using Prohorov's theorem,
we can extract a subsequence $\nu_{N_k}$ of $\nu_N$ which
converges weakly to some limit $\nu$ as $k\to\infty$. 
Note that by the construction, any limit of $\nu_N$ is an invariant distribution for the Markov chain $(\bar\omega_n)$, c.f. \cite{GZ-10} or \cite{L-82}.
Then, by the ergodic theorem and the assumption $E_\P[1/\varepsilon_1^{d+1}]<\infty$, c.f. \eqref{finite-epsilon}, we would conclude that
\[
\int gd\nu\le C\norm{g/\varepsilon_1}_{L^{d+1}(\P)} \quad\mbox{ for any 
continuous function $g$ on $\Omega^{\mathbb N}$}.
\]
The above inequality implies that $\nu$ is absolutely
continuous with respect to the probability measure $\mu$
defined by
\[
d\mu:=\frac{1}{\mathbb E_{\mathbb P}\left[
\varepsilon^{-(d+1)}_1(0,0)\right]}
\frac{1}{\varepsilon^{d+1}_1(0,0)}d\mathbb P.
\]
Since $\mu$ is by definition absolutely continuous with respect to
$\mathbb P$, we conclude that $\nu\ll\mathbb P$.
Now, note also that  Theorem~\ref{kozlov} ensures that $\nu$ is unique.

In the case in which \eqref{finite-epsilon} is satisfied
for $k>1$, by the same argument as in the case $k=1$,
we can construct an invariant measure $\nu_k$ which
is absolutely continuous with respect to $\P$,
for the environmental process
looked at times which are multiples of $k$, defined
for $n\ge 0$ by
\[
\bar\omega^{(k)}_n:=\theta_{nk,X_{nk}}\omega.
\]
We will now show how to construct from $\nu_k$ an invariant
measure $\nu$ which
is absolutely continuous with respect to $\P$, for the environmental
process $\{\bar\omega_n:n\ge 0\}$. Define for every bounded
and continuous function $g$, the measure $\nu$ by
\[
\int gd\nu:=
\frac{1}{k}\int\sum_{i=0}^{k-1} R^i gd\nu_k,
\]
where the operator $R$ is defined  by
\[
Rg(\omega):=E_{0,0,\omega}[g(\theta_{1,X_1}\omega)]=
\sum_{e\in U}\omega_0(0,e)g(\theta_{1,e}\omega)
\]
and $R^0=I$ is the identity map. 
Then note that 
\[
\int Rg d\nu
=\frac{1}{k}\int \sum_{i=1}^k R^ig d\nu_k
=\int g\dd\nu+\frac{1}{k}\int (R^k g-g)\dd\nu_k
=\int g\dd\nu,
%=\int gd\nu+\int gd\nu-\int R^k gd\nu=\int gd\nu,
\]
where in the last equality we used the fact that $\nu_k$ is an invariant distribution for the kernel $R^k$.
 This proves that $\nu$ is an invariant measure
for the environmental process. To see that $\nu$ is absolutely
continuous with respect to $\P$ note that for each
measurable $A$ in $\Omega$, with $\P(A)=0$, we have

\[
\int R^i 1_Ad\nu_k \le \sum_{z\in U_i}\nu_k(\theta_{i,z}^{-1}A)=0
\quad
\forall i=1,\ldots k,
\]
since  the stationarity
of $\P$ implies that $\P(\theta_{i,z}^{-1}A)=0$ which in turn
implies by the fact that $\nu_k\ll P$ that 
$\nu_k(\theta_{i,z}^{-1}A)=0$. Therefore we conclude that
$\nu(A)=0$ and hence $\nu$ is absolutely continuous
with respect to $\P$.

\section{Proof of the maximum principle (Theorem~\ref{maximum})}
\label{proof-maximal}
Here we will prove the maximum principle in Theorem~\ref{maximum}.
 Define
\[
M:=\max_{\mathcal D}u.
\]
Without loss of generality assume that $M>0$, $\max_{\mathcal D^p}u\le 0$, $\varepsilon_a>0$ and $f\ge 0$ in $\mathcal D$. For each $(x,n)\in\mathcal D$ define
\[
I_u(x,n):=
\left\{p\in\mathbb R^d:u(x,n)-u(y,m)\ge p\cdot (x-y)\ {\rm for}\ {\rm all}\
(y,m)\in\mathcal D\cup\mathcal D^p\text { with } m>n\right\}.
\]
Let also
\[
\Gamma=\Gamma(u,\mathcal D):=\left\{(x,n)\in\mathcal D: I_u(x,n)\ne \emptyset\right\},
\]
\[
\Gamma^+=\Gamma^+(u,\mathcal D):=
\left\{(x,n)\in\Gamma: R|p|_2<u(x,n)-p\cdot x\ {\rm for}\ {\rm some}\ 
p\in I_u(x,n)\right\}
\]
and
\[
\Lambda:=\left\{(\xi,h)\in\mathbb R^d\times\mathbb R: R|\xi|_2<h<\frac{M}{2}
\right\}\subset\mathbb R^{d+1}.
\]
For $(x,n)\in\mathcal D$ define the set

\[
\chi (x,n):=\left\{(p,q-x\cdot p):p\in I_u(x,n), q\in [ u(x,n+1),u(x,n)]\right\}
\subset\mathbb R^{d+1}.
\]

\noindent {\it Step 1.} We will first show that

\begin{equation}
\label{inclusion}
\Lambda\subset\chi(\Gamma^+)=\bigcup_{(x,n)\in\Gamma^+}\chi(x,n).
\end{equation}
Indeed, let $(\xi,h)\in\Lambda$, and define for $(x,n)\in\mathcal D$,
\[
\phi(x,n):=u(x,n)-\xi\cdot x-h.
\]
Let $(x_0,n_0)\in\mathcal D$ be such that
$u(x_0,n_0)=M$. Then, by the definition of
$\Lambda$, we see that $\phi(x_0,n_0)>0$ and
\[
\phi(x,n)<0,
\]
for $(x,n)\in\mathcal D^p$.
We now claim that there exists $(x_1,n_1)\in\Gamma^+$ with $n_1\ge n_0$
such that $\phi(x_1,n_1)\ge 0$ and $(\xi,h)\in\chi(x_1,n_1)$. Indeed,
for $x\in B_R$, let
\[
N_x:=\max\{n:(n,x)\in\mathcal D\ {\rm and}\ \phi(x,n)\ge 0\}
\]
and
\[
 n_1:=\max_{x\in B_R}N_x\ge n_0\ge 0,
\]
with the convention $\max\emptyset=-\infty$. Let $x_1\in B_R$ 
be such that $n_1=N_{x_1}$. Thus, for all $(x,n)\in\mathcal D\cup\mathcal D^p$
with $n> n_1$,
\[
u(x,n)-\xi\cdot x<h\le u(x_1,n_1)-\xi\cdot x_1.
\]
Hence $\xi\in I_u(x_1,n_1)\ne\emptyset$ and $h+\xi\cdot x_1\in
(u(x_1, n_1+1),u(x_1,n_1)]$, which proves the claim and
the statement of display \eqref{inclusion}.

\noindent {\it Step 2.} We will now show that
for each $(x,n)\in\Gamma^+$,
\begin{equation}
\label{step2}
\left|I_u(x,n)\right|
\le 
C(\#U)^d\frac{(L_a^*u(x,n))^d}{|conv(U_{x,n})|},
\end{equation}
where for every $(x,n)\in \Gamma^+$ and function
$h(x,n):\Gamma^+\to\mathbb R$ we define
\[
L^*_ah(x,n):=\sum_{z\ne 0}a_n(x,z)\left(h(x,n)-h(x+z,n+1)\right).
\]
%\[
%U_{x,n}:=\left\{a(n,x,z)z:z\in U\right\}.
%\]
Fix $p\in I_u(x,n)$ and set
\[
w(y,m):=u(y,m)-p\cdot y.
\]
Then $I_w(x,n)=I_u(x,n)-p$. 
In particular, $0\in I_w(x,n)$ and we have $w(x,n)-w(x+e,n+1)\ge 0$ for all $e\in U$.
Furthermore, if
$q\in I_w(x,n)$ and $e\in U$ 
then
\[
w(x,n)-w(x+e,n+1)\ge -e\cdot q.
\]
Hence, for each $q\in I_w(x,n)$ and $z\in U_{x,n}$ 
 we have
\begin{equation}
\label{five}
L^*_a u(x,n)=L^*_a w(x,n)
=\sum_{e\ne 0}a_n(x,e)\left(w(x,n)-w(x+e,n+1)\right)
\ge 
-z\cdot q.
\end{equation}
Recall the definition of the $U_{x,n}$ in \eqref{def-u-v}.
Let now $V_{x,n}:=conv(U_{x,n})$  and consider
the polar body of $V_{x,n}$, given by
 $V_{x,n}^o:=\{z\in\mathbb R^d:z\cdot y\le 1\
{\rm for}\ {\rm all}\ y\in V_{x,n}\}$.
Display \eqref{five} implies that
\begin{equation}
\label{six}
-I_w(x,n)\subset L^*_a u(x,n) V^o_{x,n}.
\end{equation}
Using the fact that $\sum_{l\in U_{x,n}} l=0$,
 note that if $z\in V^o_{x,n}$, then for each $y\in U_{x,n}$,
\[
z\cdot (-y)=z\cdot\sum_{l\in U_{x,n}\backslash\{y\}} l\le C(\#U).
\]
Hence, setting $\tilde U_{x,n}:=\left\{\pm y:y\in U_{x,n}\right\}$
and $\tilde V_{x,n}:=conv(\tilde U_{x,n})$ we see that
\begin{equation}
\label{seven}
V^o_{x,n}\subset\left\{z:z\cdot y\le (\#U)\ {\rm for}\ {\rm all}\ y\in \tilde
U_{x,n}\right\}=  (\#U)\tilde V^o_{x,n}.
\end{equation}
Combining (\ref{six}) with (\ref{seven}) we conclude that
\[
 -I_w(x,n)\subset  (\#U) L^*_a u(x,n) \tilde V^o_{x,n}.
\]
Now, since $\tilde V^o_{x,n}$ is a symmetric convex body,
by Mahler's inequality \cite{M-39}, we see that
\[
|\tilde V^o_{x,n}|\le \frac{4^d}{|\tilde V_{x,n}|},
\]
which finishes the proof of (\ref{step2}).

\noindent {\it Step 3.} Here we derive the maximum inequality
from steps 1 and 2. Set
\[
\chi(\Gamma^+,x):=\bigcup_{m:(x,m)\in\Gamma^+}\chi(x,m).
\]
For each $x\in\mathcal D$, define
$\rho_x:\mathbb R^d\to\mathbb R^d$ by
\[
\rho_x(y,m)=(y,m+y\cdot x)
\]
and let
\[
\tilde\chi(x,n):=\rho_x\circ\chi(x,n)=I_u(x,n)\times 
[u(x, n+1),u(x,n)]
\subset\mathbb R^{d+1}.
\]
Then, using the inequality $a^{\frac{1}{d+1}}b^{\frac{d}{d+1}}\le \frac{a+db}{d+1}$,
valid for $a\ge 0$, $b\ge 0$, and the notation $\sum'$ for
the sum running from $n=1$ to $n=T$ with $u(x,n)-u(x, n+1)$
and $L^*_au(x,n)$ positive
we see that
\begin{align}\label{last}
\left|\chi(\Gamma^+,x)\right|
&=
\left|\tilde\chi(\Gamma^+,x)\right|\nonumber\\
&\le 
\sideset{}{'}\sum 
(u(x,n)-u(x,n+1))\left|I_u(x,n)\right| 1_{(x,n)\in\Gamma^+}\nonumber\\
&\le 
(\# U)^d\sideset{}{'}\sum 
(u(x,n)-u(x,n+1))\frac{\left(L^*_au(x,n)\right)^d}{v(x,n)} 1_{(x,n)\in\Gamma^+}\nonumber\\
&\le  
C(\# U)^d\sideset{}{'}\sum 
\left(\frac{a_n(x,0)(u(x,n)-u(x,n+1))+L^*_au(x,n)}{(d+1)
\varepsilon(x,n)}\right)^{d+1}1_{(x,n)\in\Gamma^+}\nonumber\\
&=
C(\# U)^d\sideset{}{'}\sum
\left(\frac{-\mathcal L_au(x,n)}{\varepsilon(x,n)}\right)^{d+1}1_{(x,n)\in
\Gamma^+}.
\end{align}
Now, note that
\[
|\Lambda|=C \frac{M^{d+1}}{R^d}.
\]
Combining this with inequalities \eqref{inclusion}, \eqref{last} and
using the hypothesis \eqref{poisson-inequality}, we see that
\[
C \frac{M^{d+1}}{R^d}\le\sum_{(x,n)\in\mathcal D}\frac{1}{\varepsilon^{d+1}}
|f|^{d+1} 1_{(x,n)\in\Gamma^+}.
\]
Therefore,
\[
\max_{(x,n)\in\mathcal D}u(x,n)\le C R^{\frac{d}{d+1}}\left\|\frac{f}{\varepsilon}
\right\|_{\mathcal D,d+1}.
\]

\section{Proof of the continuous time QCLT (Theorem~\ref{qclt-continuous})}
\label{section-continuous}
The proof of Theorem \ref{qclt-continuous} follows a strategy similar to that
of Theorem \ref{qclt}. In other words, since the continuous time
random walk is also $\mathbb Q$-a.s. a martingale, it suffices to construct an invariant measure for the environmental
process which is absolutely continuous with respect to
the initial law $\mathbb Q$ of the environment. 
%%%GGG<
However, unlike the discrete time case, the continuous time process is allowed to jump at unbounded rates. To obtain a QCLT, we need not only to deal with the degeneracy of the ellipticity, but also to control the jump rates. 
This is achieved by first performing a time change to ``slow-down" the original RWRE, and then applying a maximum principle for (continuous-time) parabolic difference operators to construct the invariant measure.
%The construction
%of such an invariant measure follows a strategy similar to
%the one presented in Section \ref{proof-main}, where we require
%a parabolic maximum principle.
%%%GGG>

Let us state the version of the parabolic maximum principle that we use. Consider a balanced 
continuous time-dependent environment
$\{a_t:t\ge 0\}$, c.f. \eqref{bcte}, with
$a_t:=\{a_t(x):x\in\mathbb Z^d\}$ and
$a_t(x):=\{a_t(x,e):e\in U\}\in\mathcal Q$.  
%%%GGG<: The definition of connectedness is not necessary and is commented out.
%We say that a subset $D$ of $\mathbb Z^d$ is connected
%if any two point of the set can be joined by a sequence
%of points in $D$ such that two consecutive members of the
%sequence are nearest neighbors. 
%%%GGG>
Given any finite set $D\subset \mathbb Z^d$ and $T>0$, we define 

\[
\mathcal D:=D\times[0,T).
\]
%%%GGG<
Define the {\it parabolic boundary} of $\mathcal D$ by 
\[
\mathcal D^p:=\mathcal D^\ell\cup\mathcal D^T,
\]
where $\mathcal D^T=D\times\{T\}$ denotes its {\it time boundary} and 
\[\mathcal D^\ell:=\{(x,t)\notin\mathcal D: a_t(y,x-y)>0 \mbox{ for some }(y,t)\in\mathcal D\}
\] is the {\it lateral boundary} of $\mathcal D$.
%Let $\bar{\mathcal D}:=\bar D\times [0,T]$ and define
%the {\it parabolic boundary} of $\mathcal D$
%by
%
%$$
%\mathcal D^p:=\bar{\mathcal D}\backslash\mathcal D.
%$$
%%%GGG>

%Now, for $u:\mathbb Z^d\times [0,\infty)\to
%\mathbb R$ bounded and differentiable in time for each $x\in\mathbb Z^d$, we define the parabolic difference operator 
%%%%GGG:
%\[
%\mathcal L_a u(x,t):=\sum_{e\in U}a_t(x,e)[u(x+e,t)-u(x,t)]{\blue +\partial_t u(x,t).}
%\]
For $p>0$ and any real-valued function $g$ that is summable on $\mathcal D$, define
\[
\norm{g}_{\mathcal D,p}:=
\left(\int_0^T\sum_{x\in D}|g(x,t)|^p\dd t\right)^{1/p}.
\]
We  can now state the maximum principle.

\begin{theorem}
\label{max-cont} Assume that $a$ is a balanced environment. Let $u$
be  a function on $\mathcal D\cup\mathcal D^p$ which is differentiable with respect
to $t$ in $(0,T)$. Let $f$ be an integrable function in $\mathcal D$.
Assume that $u$ satisfies
\[
\mathcal L_a u\ge f\qquad {\rm in}\quad\mathcal D.
\]
 Then, there is a constant $C=C(U,d)>0$
such that
\[
\sup_{\mathcal D}u\le \sup_{\mathcal D^p}u+CR^{d/(d+1)}||f/\varepsilon||_{\mathcal D,d+1},
\]
where $R:=diam(D)$ and $\varepsilon$ is as defined in \eqref{epsilon-cont}.

\end{theorem}

%%%GGG<
Recall that the space-time process $(X_t,t)_{t\ge 0}$ is a Markov process on $\Z^d\times\R$ with generator $\mathcal L_\omega$.
%\[
%\mathcal L^*_\omega u(x,t)
%:=\sum_{e\in U}\omega_t(x,e)(u(x+e,t)-u(x,t))+\partial_t u(x,t).
%\]
To show that $(X_t)_{t\ge 0}$ does not explode, i.e, there are only finitely many jumps within finite time, we will first consider a slowed-down process. 
%One may worry that $X_t$ travels too fast and thus we first slow it down. 
Recall the definition of $\upsilon_\omega$ in \eqref{epsilon-cont}.
Let 
\[
(Y_t, T_t)_{t\ge 0}
\] 
be the Markov process on $\Z^d\times\R$ with generator $(\upsilon_\omega+1)^{-1}\mathcal L_\omega$ and initial state $(Y_0,T_0)=(0,0)$. Note that the process $(Y_\cdot,T_\cdot)$ has slower transition rates on both the $\Z^d$-coordinate and the $\R$-coordinate, compared to $(X_t,t)$.
Note also that 
\begin{equation}\label{eq-T}
T_t=\int_0^t\frac{1}{\upsilon_\omega(Y_s,T_s)+1}\dd s
\end{equation}
and
\begin{equation}\label{X-T}
X_{T_t}\stackrel{d}=Y_t.
\end{equation}
%%%GGG<:
Define the stopping times $\tau_0=\tau_0(Y_\cdot,T_\cdot)=0$, and
\[
\tau_{j+1}=\tau_{j+1}(Y_\cdot,T_\cdot)
=\inf\{t>\tau_j:|Y_t-Y_{\tau_j}|_\infty>N \mbox{ or } T_t-\tau_j>N^2\}.
\] 

% For any $\omega\in\Omega$, define $\bar\omega\in\Omega$ for $t\ge 0$, $x\in\mathbb Z^d$ and $e\in U$ by
%%%GGG>
With abuse of notation, we enlarge the probability space and still use $P^c_{0,0,\omega}$  to denote the joint law  of $X_\cdot$, the environmental process $\bar\omega_\cdot$ and the process $(Y_\cdot,T_\cdot)$ with initial state $(0,0)$. We let $E^c_{0,0,\omega}$ denote the expectation under $P^c_{0,0,\omega}$.
We have the following analogue of Lemma \ref{L-one}.

\begin{lemma}\label{L-two}
There exists a constant $c>0$ such that for all $N$ large and any $\omega\in\Omega$,  
\[
E^c_{0,0,\omega}\bigg[\big(1-\dfrac{c}{N^2}\big)^{\tau_1(Y_\cdot, T_\cdot)}\bigg]\le \frac{1}{2}.
\]
\end{lemma}
%%%GGG<:
\pf The proof follows similar argument as in Lemma~\ref{L-one}. 
Recall that $C_U=\max\{|e|:e\in U\}$. Note that $(Y_t)_{t\ge 0}$ is a martingale and 
$(|Y_t|^2-C_Ut)_{t\ge 0}$ is a super-martingale. Let $K=\frac{N^2}{2C_U}$. Then, by Doob's $L^2$-martingale inequality,
\begin{align*}
P^c_{0,0,\omega}(\tau_1\le K)
&={P^c_{0,0,\omega}}(\max_{0<t\le K}|Y_t|>N)
\\
&\le \frac{1}{N^2}
E^c_{0,0,\omega}
[|Y_K|^2]
\le \frac{C_UK}{N^2}=\frac{1}{2},
\end{align*}
where in the first equality we used the fact that $T_K\le K<N^2$.\qed
%%%GGG>

\begin{theorem}\label{inv-meas-Y}
Assume the same conditions as in Theorem~\ref{qclt-continuous}.  Then the environmental process $(\theta_{T_t,Y_t}\omega)_{t\ge 0}$ has a unique invariant probability measure $\bar\nu$ which is equivalent to $\Q$.
\end{theorem}
%%%GGG<:
\pf Let $Q_N=\{z\in\Z^d:|z|_\infty\le N\}\times [0,N^2)$. 
  We introduce on $\Z^d$ the equivalent classes
\[
\overline{(x,t)}:=(x,t)+(2N+1)\Z^d\times N^2\Z.
\]
Fix a balanced environment $\omega\in\mathcal Q$, and define its periodized environment $\omega^{(N)}$ so that for any $(x,t)\in Q_N$, 
\[
\omega^{(N)}_s(y)=\omega_t(x)
\]
whenever $\overline{(y,s)}=\overline{(x,t)}$.

Set 
\[
\Omega_{N,\omega}=\Omega_N:=\{\theta_{t,x}\omega^{(N)}: (x,t)\in Q_N\}
\]
and let $\P_N=\P_{N,\omega}$ denote the probability measure  
\[
\P_N(\dd\xi)=
\frac{1}{N^{d+2}}\sum_{x:(x,t)\in Q_N}1_{ \theta_{t,x}\omega^{(N)}=\xi}\dd t.
\]
Under the environment $\omega^{(N)}$, recall that $Y_t$ is the slowed-down process. 
Since under $P^c_{0,0,\omega^{(N)}}$, the process $\overline{(Y_t,T_t)}_{t\ge 0}$ is a Markov process on the compact set $Q_N$, it has an invariant distribution whose  density we denote by $\phi_N(x,t)\delta_x\dd t$, with $(x,t)\in Q_N$ and $\delta_x$ denotes the Dirac mass. As in the proof of Theorem~\ref{qclt}, the probability measure $\nu_N\ll \P_N$ defined by 
\[
\nu_N(\dd\xi)=\sum_{x:(x,t)\in Q_N}\phi_N(x,t)1_{\theta_{t,x}\omega^{(N)}=\xi}\dd t,
\]
is an invariant distribution of the Markov process $(\theta_{T_t,Y_t}\omega^{(N)})_{t\ge 0}$.
%By compactness, there is an invariant probability measure $\nu_N\ll \P_N$ for the Markov chain $(\theta_{T_t,Y_t}\omega^{(N)})_{t\ge 0}$, and it has the form
%\[
%\nu_N(\dd\xi)=\sum_{x:(x,t)\in Q_N}\phi_N(x,t)1_{\theta_{t,x}\omega=\xi}\dd t,
%\]
%where $\phi_N$ is the density of the invariant measure of the random walk $(\overline{Y_t,T_t})_{t\ge 0}$ on $Q_N$.

For $\xi\in\Omega_N$, let $\xi_t:=\theta_{T_t,Y_t}\xi$ denote the environmental process.
By similar arguments as in Section~\ref{proof-main}, Lemma~\ref{L-two} implies that
for any bounded continuous function $g$ on $\Omega^\N$,
\[
\int g \dd\nu_N\le CN^{-2}\max_{\xi\in\Omega_N}E_{0,0,\xi}^c\left[\int_0^{\tau_1}g(\xi_t)\dd t\right].
\]
%{\red Referee on $\nu_N g$: ``you have introduced a new notation here"\\}
Letting 
\[
u(x,t)=E_{0,0,\theta_{t,x}\xi}^c\left[\int_0^{\tau_1}g(\xi_s)\dd s
%\bigg|Y_0=x, T_0= t
\right],
\]
we have
\[
\left\{
\begin{array}{rl}
(\upsilon_\omega+1)^{-1}\mathcal L_\xi u
=-g(\theta_{t,x}\xi)
&\mbox{in }Q_N\\
u=0\quad &\mbox{in }Q_N^p,
\end{array}
\right.
\]
Then, applying Theorem~\ref{max-cont} to the operator $\mathcal L_\xi$, we get 
\[
\max_{Q_N}u\le CN^2\norm{(\upsilon+1)g/\varepsilon}_{L^{d+1}(\P_N)}
\]
and so 
\[
\int g\dd\nu_N\le C\norm{(\upsilon+1)g/\varepsilon}_{L^{d+1}(\P_N)}.
\]
Since, $\lim_{N\to\infty}\P_{N,\omega}=\Q, \,\Q$-a.s. and 
\[E_\Q[(\upsilon+1)^{d+1}/\varepsilon_\omega^{d+1}]
\le 
2^dE_\Q[(\upsilon^{d+1}+1)/\varepsilon_\omega^{d+1}]<\infty,
\] using the ergodic theorem and Kozlov's argument, the conclusion follows.\qed 

\begin{corollary}
Assume the same conditions as in Theorem~\ref{qclt-continuous}.
For $\Q$-almost all $\omega$, $P_{0,0,\omega}^c$-almost surely the process $(X_t)_{t\ge 0}$ does not explode. Moreover, the environmental process $(\theta_{t,X_t}\omega)_{t\ge 0}$ has a unique invariant probability measure $\nu$ which is equivalent to $\Q$.
\end{corollary}
\pf 
Set $\bar\omega_t:=\theta_{t,X_t}\omega$. 
By \eqref{eq-T}, Theorem~\ref{inv-meas-Y} and the ergodic theorem, 
\[
\lim_{t\to\infty}\frac{T_t}{t}
=E_{\bar\nu}[\frac{1}{\upsilon+1}]\in(0,1) \qquad \mbox{$\Q\otimes P_{0,0,\omega}$-a.s}.
\]
Hence, by \eqref{X-T}, $(X_t)_{t\ge 0}$ is not explosive. 
Furthermore, let 
\[
\dd\nu:=\frac{N}{\upsilon+1}\dd\bar\nu,
\]  where $N=(E_{\bar\nu}[\frac{1}{\upsilon+1}])^{-1}$ is a normalization constant. Then $\nu$ is the invariant measure of $(\bar\omega_t)_{t\ge 0}$ and it is equivalent to $\Q$.\qed
%%%GGG>
%%%This is where I stopped.

Theorem \ref{qclt-continuous} (i) is proved in the above corollary. As in Theorem~\ref{qclt}, this implies the invariance principle Theorem \ref{qclt-continuous} (ii).
% and 
%\[
%\int gd\bar\nu \le C\norm{g/\varepsilon_{\bar\omega}}_{L^{d+1}(\Q)}.
%\]
%Noting that 
%\begin{align*}
%\int \frac{1}{\upsilon}d\bar\nu
%&\le C(E_\Q[(\upsilon\varepsilon_{\bar\omega})^{-(d+1)}])^{1/(d+1)}\\
%&=
%C\bigg(E_\Q[\frac{1}{\upsilon\varepsilon^{d+1}}]\bigg)^{1/(d+1)}<\infty,
%\end{align*}

\section{Proof of Corollary \ref{qclt-static}}
\label{section-application}
Recall the definitions of $x^{(1)}, x^{(2)}$ before Corollary~\ref{qclt-static}. We set 
\[
Y_n:=(X^{(1)}_n,0)\in\Z^{d_1+d_2} \quad\text{and }
Z_n:=(0,X^{(2)}_n)\in\Z^{d_1+d_2},
\] 
so that $X_n=Y_n+Z_n$. For each $n\ge 0$, denote by
$\mathcal F^Y$ the $\sigma$-algebra generated
by $\{Y_0,Y_1,\ldots\}$. Furthermore, we define a time-dependent environment $\omega^Y$ on $\mathcal P^{\Z^{d_2}}$ by
\[
\omega^Y_n(z,e)
:=
\frac{\omega\left(Y_n+(0,z), Y_{n+1}-Y_n+(0,e)\right)}
{\omega(Y_n,Y_{n+1}-Y_n)}
\quad
\text{ for }z,e\in\Z^{d_2}, n\in\N.
\]
%Now, given $\omega\in\Omega$,
%define the time dependent environment $\omega^Y:=\{\omega^Y_n:n\ge 0\}$
%on $\Omega=\mathcal P^{\mathbb Z^d}$ by
%
%$$
%\omega^Y_n:=\{\omega^Y_n(x):x\in\mathbb Z^d\}
%$$
%for $n\ge 0$ and $\omega^Y_n(x):=\{\omega^Y_n(x,e):e\in\{e_2,-e_2,0\}\}$ with
%\[
%\omega^Y_n((y,z),e):=\omega_{D_2}((y+Y_n,z),e)
%\]
%for $y,z\in\mathbb Z$ and $f\in\{e_2,-e_2,0\}$.
%Furthermore, we define the time dependent environment $\tilde \omega^Y:=
%\{\tilde\omega^Y_n:n\ge 0\}$ on $\Omega_1:=\mathcal P_2^{\mathbb Z}$
%with $\mathcal P_2:=\{v(e), e\in\{e_2,-e_2,0\}:v(e)\ge 0,\sum_{e\in\{e_2,-e_2,0\}}
%v(e)=1\}$,
%by
%
%$$
%\tilde\omega^Y_n:=\{\tilde\omega_n^Y(y):y\in\mathbb Z\}
%$$
%for $n\ge 0$ and $\tilde\omega^Y_n(y):=\{\tilde\omega_n^Y(y,e):e\in\{e_2,-e_2,0\}\}$
%with
%
%$$
%\tilde\omega^Y_n(y,e):=\omega^Y_n((0,y),e).
%$$

\begin{lemma}\label{2nd coord} $\P\times P_{0,\omega}$-a.s. under the law $P_{0,\omega}(\cdot|\mathcal F^Y)$, $\{X^{(2)}_n:n\ge 0\}$ is a random walk
on the lattice $\mathbb Z^{d_2}$
in the time dependent environment $\omega^Y$.
\end{lemma}
\begin{proof} 
%For $y\in\Z^{d_1}, z\in\Z^{d_2}$, we let $\check y:=(y,0)\in\Z^{d_1+d_2}$ and $\hat z:=(0,z)\in\Z^{d_1+d_2}$.
For $0\le m<n-1$ and any two sequences
$z_1,\ldots,z_{m+1}\in\mathbb Z^{d_2}$ and
$y_1,\ldots,y_n\in\mathbb Z^{d_1}$,
we let $x_i=(y_i,z_i)$ for $1\le i\le m+1$. 
%Using
%property $(b)$ of Theorem \ref{qclt-static} (independence property of the
%jumps of the coordinates, c.f. (\ref{independence-prop}))
%and property $(a)$ of Theorem \ref{qclt-static} (the fact that the
%first coordinate of the random walk is autonomous, c.f. (\ref{autonomous-prop})), note that
By the Markov property, 
\begin{align*}
&P_{0,\omega}(X^{(2)}_{m+1}=z_{m+1}|X^{(2)}_1=z_1,\ldots,X^{(2)}_m=z_m,X^{(1)}_1=y_1,\ldots,X^{(1)}_n=y_n)\\
&=
\frac{P_{0,\omega}(X_1=x_1,\ldots, X_{m+1}=x_{m+1}, X^{(1)}_{m+2}= y_{m+2},\ldots, X^{(1)}_n=y_n)}
{P_{0,\omega}(X_1=x_1,\ldots, X_m=x_m, X^{(1)}_{m+1}=y_{m+1},\ldots, X^{(1)}_n=y_n)}\\
&=\frac{P_{0,\omega}(X^{(1)}_{m+2}=y_{m+2},\ldots, X^{(1)}_n=y_n|X_{m+1}=x_{m+1}) P_{0,\omega}(X_{m+1}=x_{m+1}|X_m=x_m)}
{P_{0,\omega}(X^{(1)}_{m+1}= y_{m+1},\ldots, X^{(1)}_n=y_n|X_m=x_m)}\\
%&=\frac{\omega_Z(x_m,z_{m+1}-z_m)\omega_Y(x_m,y_{m+1}-y_m)
%\bar\omega_Y(y_{m+1},y_{m+2}-y_{m+1})\cdots
%\bar\omega_Y(y_{n-1},y_{n}-y_{n-1})}
%{\omega_Y(x_m,y_{m+1}-y_m)
%\bar\omega_Y(y_{m+1},y_{m+2}-y_{m+1})\cdots
%\bar\omega_Y(y_{n-1},y_{n}-y_{n-1})}\\
%&=\omega_Z(x_m,z_{m-1}-z_m)\\
%&=\omega_Z((y_m,z_m),z_{m-1}-z_m),
&=
\frac{\omega(x_m,x_{m+1}-x_m)}{P_{x_m,\omega}(X_1^{(1)}=y_{m+1})},
\end{align*}
where in the last equality we used condition (a) which says that $X^{(1)}_n$ is a Markov chain.  We finish the proof by observing that
\[
\frac{\omega(x_m,x_{m+1}-x_m)}{P_{x_m,\omega}(X_1^{(1)}=y_{m+1})}
=
\omega^Y_m(z_m,z_{m+1}-z_m)\big|_{Y_i=y_i,i=1,\ldots,n}.
\]
\end{proof}

\noindent{\it Proof of Corollary~\ref{qclt-static}:}

Let $\Omega=\mathcal P^{\Z^{d_1+d_2}}$.
We will only consider the non-trivial case where $\P$-almost surely, $\theta_{e_i}\o\neq\o$ for all $i=1,\ldots,d_1$. (Otherwise, since $\{\omega:\theta_{e_i}\o=\o\}$ is shift-invariant under all shifts $\{\theta_x:x\in\Z^{d_1+d_2}\}$, it follows by ergodicity that $\P(\theta_{e_i}\o=\o)=1$ for some $i\in\{1,\ldots,d_1\}$, which then implies that every measurable set $A\subset\Omega$ is shift-invariant under the shift $\theta_{e_i}$. By our ergodicity assumption, we conclude that $\P$ is a singleton, i.e, $\P(\omega=\xi)=1$ for some $\xi\in\Omega$. In this case, the RWRE is a simple random walk and the QCLT is trivial.)

For any $z\in\Z^{d_2}$, we denote by $\hat z:=(0,z)\in\Z^{d_1+d_2}$ so that $\hat z^{(2)}=z$. 
Our proof contains several steps.
\begin{enumerate}[\text{Step }1.]
\item
Set $\tilde\omega_n:=\theta_{Y_n}\omega$. 
By the ergodic theorem, the measure $\nu$ that satisfies the properties in condition (b) of Corollary~\ref{qclt-static} is unique.  
Let us denote by $Q_\nu$ the law of $(\tilde\omega_n)_{n\ge 0}$
starting from $\nu$. 

\item 
We will show that the law $Q_\nu$ of the space-time environment 
$\{\tilde\omega_n(x): (x,n)\in \Z^{d_1+d_2}\times\N\}$
 is translation invariant under the spatial shifts $\{\theta_{0,\hat z}:z\in\Z^{d_2}\}$. 
To this end, it is enough to prove that $\nu$, as a measure on the static environments $\Omega$, is translation invariant under these spatial shifts.
 Indeed, since by condition (a), the law of $\{X^{(1)}_n:n\in\N\}$ under $P_{x,\omega}$ depends only on the first $d_1$ coordinates of $x$, we conclude that for any $z\in\Z^{d_2}$, the law $\hat\nu_z$ defined by $\hat\nu_z(A):=\nu(\theta_{\hat z}A)$ is still an invariant measure for the Markov chain $(\theta_{Y_n}\omega)_{n\in\N}$. Furthermore, by the stationarity of $\P$ under the spatial shifts, $\hat\nu_z$ is also equivalent to $\P$.  Therefore by the uniqueness of $\nu$,  we have $\hat\nu_z=\nu$ for any $z\in\Z^{d_2}$ and so $Q_\nu$ is  translation invariant under $\{\theta_{0,\hat z}:z\in\Z^{d_2}\}$.
\item
Next, we claim that $\P\times P_{0,\omega}$-almost surely, $\omega_n^Y$ can be written 
%Next, we claim that when the initial environment $\omega$ has law $\nu$,  the law of  the space-time environment $\omega^Y\in\mathcal P^{\Z^{d_2}\times\N}$ is
%ergodic under the time shifts and translation-invariant under the spatial shifts $\{\theta_{1,z}:z\in\Z^{d_2}, |z|\le 1\}$. Indeed, since $\P$ is ergodic under $\{\theta_{\pm e_i}\}$ for every $i=1,\ldots, d_1$, we can write $Y_{n+1}-Y_n$ ($\P\times P_{0,\omega}$-a.s. and also $Q_\nu$-a.s.) 
as a function of $\tilde\omega_n$ and $\tilde\omega_{n+1}$, $n\ge 0$.
Indeed, for any $n\ge 0$,
\begin{align*}
Q_\nu\left(\tilde\o_n=\tilde\o_{n+1}, Y_n-Y_{n+1}\neq 0\right)
&\le 
Q_\nu\left(\theta_{e_i}\tilde\o_n=\tilde\omega_n\,\text{ for some }i\in\{1,\ldots,d_1\}\right)\\
%&=
%Q_\nu\left(\sum_{|e|\le 1}\mathbb 1_{\theta_{(e-Y_1)}\o=\omega}\ge 2\right)\\
&=
\nu\left(\theta_{e_i}\o=\o \,\text{ for some }i\in\{1,\ldots,d_1\} \right)\\
&=0,
\end{align*}
where in the first equality we used that $(\tilde\o_n)_{n\ge 0}$ is a stationary sequence under $Q_\nu$ and in the last equality we used $\nu\approx\P$ and the assumption at the beginning of the proof.  Hence, the events $\{Y_n= Y_{n+1}\}$ and $\{\tilde\o_n=\tilde\o_{n+1}\}$ are equivalent.
 In particular, we write for $z,e\in\Z^{d_2}$,
\begin{equation}\label{yomega}
\omega^Y_n(z,e)
=
\frac{\tilde\omega_n(\hat z,\hat e)}{\tilde\omega_n(0,0)}
\mb 1_{\tilde\omega_n=\tilde\omega_{n+1}}+\mb 1_{e=0, \,\tilde\o_n\neq\tilde\o_{n+1}}.
\end{equation}
%Therefore, using the properties of $(\tilde\omega_n)_{n\ge 0}$ obtained in Step 2, our claim is proved.
% by Step 2 we conclude that when $\omega$ has law $\nu$, then $\omega^Y$ is ergodic with respect to {\blue $\{\theta_{1,z}:z\in\Z^{d_2}, |z|\le 1\}$}.
\item
Notice that $\omega^Y$ is not an elliptic environment. However, we will show that under a time change, $X_n^{(2)}$ is a random walk in an ergodic uniformly-elliptic random environment (conditioning on $Y_\cdot$). 
Indeed, let $D=\{\xi=(\xi_i)_{i\ge 0}\in\Omega^\N:\xi_0=\xi_1\}$ and
\[
\phi_D(\tilde\o)=\inf\{n\ge 0:\theta_{n,0}\tilde\o\in D\}=\inf\{n\ge 0: \tilde\omega_n=\tilde\omega_{n+1}\}.
\]
Since $Q_\nu(\tilde\omega\in D)=E_{Q_\nu}[P_{0,\o}(Y_1=0)]>0$ and the law $Q_\nu$ of $\tilde\omega$ is ergodic under the time shift $\theta_{1,0}$, by ergodicity, $\phi_D<\infty$ almost surely. Moreover, defining the induced shift $T_D:\Omega^\N\to\Omega^\N$ as
\[
T_D\tilde\omega:=\theta_{\phi_D(\tilde\omega),0}\tilde\omega,
\]
then $(T_D^k\tilde\omega)_{k\ge 1}$ (under law $Q_\nu$) is still an ergodic sequence, c.f. \cite[Theorem~1.6]{Sar-09}.
Recall that by \eqref{yomega},  $\omega_n^Y$ is a function of $\tilde\omega_n$ and $\tilde\omega_{n+1}$. Hence, by ergodicity of the sequence $(\tilde\omega_n)_{n\ge 0}$, the time-dependent environment $\zeta^{Y}\in\mathcal P^{\Z^{d_2}\times\N}$ defined by
\[
\zeta^Y_n(x,e):=(T_D^{n+1}\o^Y)_0(x,e), \quad x\in\Z^{d_2}, n\ge 0
\]
is ergodic under time-shifts and stationary (by Step 2) under the spatial shifts.
We thus conclude that $\zeta^Y$ is a uniformly-elliptic balanced time-dependent random environment  which (conditioning on $Y_\cdot$ and under $Q_\nu$) is ergodic with respect to the space-time shifts $\{\theta_{1,z}: |z|\le 1, z\in\Z^{d_2}\}$. Furthermore, define recursively random times $\phi_0:=\phi_D$ and
\[
\phi_{i+1}(\tilde\omega):=\inf\{n>\phi_i: \tilde\omega_n=\tilde\omega_{n+1}\}.
\]
%(Note that given the path $Y_\cdot$ and $\omega$, $(\phi_n)_{n\ge 0}$ is a deterministic sequence.) 
Then, conditioning on $Y_\cdot$,
\begin{equation}\label{timechange}
W'_n:=X_{\phi_n(\tilde\omega)}^{(2)}, \quad n\ge 0
\end{equation}
is a random walk in the time-dependent environment $\zeta^Y$ defined above.

\item 
Now, we will prove a QCLT for $X^{(2)}_n$. First, since for $\tilde\omega$ sampled according to $Q_\nu$, the sequence  $W_n'$ is a random walk in a uniformly elliptic and ergodic (with respect to $\{\theta_{1,z}: |z|\le 1,z\in\Z^{d_2}\}$) balanced environment with jump range $\{z\in\Z^{d_2}: |z|\le 1\}$, by Theorem~\ref{qclt}, we obtain for $W_n'$ a QCLT with non-degenerate $d_2\times d_2$ covariance matrix. Then, noticing that by Kac's formula, for $Q_\nu$-almost every $\tilde\o$,
\[
\lim_{n\to\infty} \frac{\phi_n(\tilde\o)}{n}
=\frac{1}{Q_\nu(\tilde\o\in D)}
=\frac{1}{E_{Q_\nu}[P_{0,\o}(Y_1=0)]}<\infty,
\]
with a standard time-change argument, we conclude a QCLT for $X_n^{(2)}$ (conditioning on $Y_\cdot$). That is, 
$\P$-a.s,  for almost all trajectories $\{Y_0,Y_1,\ldots \}$
and any open $B\in C([0,\infty);\mathbb Z^{d_2})$,
\begin{equation}
\label{qclt-cond}
\lim_{n\to\infty}P_{0,\omega}\left(\left.\frac{X^{(2)}_{[n\cdot]}}{\sqrt{n}}\in B\right|\mathcal F^Y
\right)= Q(B),
\end{equation}
where $Q$ denotes the law of a
Brownian motion on $\R^{d_2}$ with a deterministic non-degenerate covariance matrix.
\item 
With condition (c), we can now conclude that for any pair of 
open sets $A\in C([0,\infty);\mathbb Z^{d_1})$
and $B\in C([0,\infty);\mathbb Z^{d_2})$,
\[
P_{0,\omega}\left(\frac{X^{(1)}_{[n\cdot]}-v_1 n\cdot}{\sqrt{n}}\in A,
\frac{X^{(2)}_{[n\cdot]}}{\sqrt{n}}\in B\right)
=
E_{0,\omega}\left[\mathbb 1_{X^{(1)}_{[n\cdot]}-v_1 n\cdot/\sqrt n\in A}
P_{0,\omega}\left(\left.\frac{X^{(2)}_{[n\cdot]}}{\sqrt{n}}\in B\right|\mathcal F^Y
\right)\right],
\]
which by (\ref{qclt-cond}) converges to 
the probability that a Brownian motion (with a deterministic non-degenerate
covariance matrix) in $\mathbb R^{d_1+d_2}$ belongs to $A\times B$. Using the
fact that any open set in $C([0,\infty):\mathbb Z^{d_1+d_2})$
is a countable union of sets of the
form $A\times B$,  we conclude the proof. 
\qed
\end{enumerate}

\section{
An example that QCLT fails when the environment is not ergodic enough
%Counterexample of a random walk in an environment which
%is not ergodic enough
}

\label{section-counterexample}
Here we show that the QCLT could fail if the ergodicity hypothesis of Theorem~\ref{qclt} is weakened. 
%Here we show that the ergodicity hypothesis of Theorem \ref{qclt}
%cannot be weakened.

Consider a discrete time-dependent balanced random environment on $\Z^2$.  
Let 
$p=\{p(e):|e|_1=1\},\, q=\{q(e):|e|_1=1\}$ be two probability vectors such that
$p(e)=\frac{1}{4}$ for all $e\in\Z^2$ with $|e|_1=1$, and $q(\pm e_1)=\frac{1}{6},\, q(\pm e_2)=\frac{1}{3}$.
For any $(x,n)\in\Z^2\times\Z$, define two space-time environments $\xi, \xi'$ such that
\[
\xi_n(x,e)=\left\{
\begin{array}{rl}
p(e) &\mbox{ if $|x|_1+n$ is even}\\
q(e) &\mbox{ if $|x|_1+n$ is odd}
\end{array}
\right.
\]
\[
\xi'_n(x,e)=\left\{
\begin{array}{rl}
p(e) &\mbox{ if $|x|_1+n$ is odd}\\
q(e) &\mbox{ if $|x|_1+n$ is even}.
\end{array}
\right.
\]
Define the environment measure $\P$ to be
\[
\P(\omega=\xi)=\P(\omega=\xi')=\frac{1}{2}.
\]
Noting that $\theta_{1,0}\xi=\xi'$ and the jump range $U=\{e\in\Z^d:|e|=1\}$. The measure $\P$ is 
ergodic
under the shifts $\{\theta_{1,e}:|e|\le 1\}$. However, $\mb P$ is not ergodic under the set of shifts $\{\theta_{1,e}: e\in U\}$ and the QCLT fails, since $X_{n\cdot}/\sqrt n$ converges to a Brownian motion with a random covariance matrix
\[
\Sigma(\omega)=\begin{pmatrix}
1/2&0\\0&1/2
\end{pmatrix}1_{\omega=\xi}
+
\begin{pmatrix}
1/3&0\\0&2/3
\end{pmatrix}1_{\omega=\xi'}.
\]
%\smallskip
%\begin{rmk}
%%Thanks to the lack of stay-put, the shifts $\{\theta_{n,X_n}:n\in\mathbb N\}$, which do not contain $\theta_{1,0}$, can only shift environments to sites with the same space-time parity, and hence does not benefit from the full strength of the ergodicity with respect to $\{\theta_{1,x}:x\in\mathbb Z^d\}$.  
%The reason why Theorem \ref{qclt} cannot be applied and
%the quenched central limit theorem fails is that $\P$ is not ergodic under the group of shifts $\{\theta_{1,x}: x\in U\}$.
%\end{rmk}

\end{document}